\documentclass[12pt]{article}
\usepackage{amsmath, amscd, amssymb, latexsym, epsfig, color, amsthm, tikz, array}
\numberwithin{equation}{section}

\newtheorem{theorem}{Theorem}[section]
\newtheorem{proposition}[theorem]{Proposition}

\newtheorem{conjecture}[theorem]{Conjecture}
\newtheorem{lemma}[theorem]{Lemma}

\newtheorem{algorithm}[theorem]{Algorithm}

\theoremstyle{definition}
\newtheorem{definition}[theorem]{Definition}
\newtheorem{example}[theorem]{Example}

\DeclareMathOperator{\type}{type}

\begin{document}

\title{Pure $\mathcal{O}$-sequences arising from $2$-dimensional PS ear-decomposable simplicial complexes}
\author{Steven Klee\\
\small Department of Mathematics \\[-0.8ex]
\small Seattle University\\[-0.8ex]
\small Seattle, WA 98122, USA\\[-0.8ex]
\small \texttt{klees@seattleu.edu}
\and Brian Nugent\\
\small Department of Mathematics \\[-0.8ex]
\small Seattle University\\[-0.8ex]
\small Seattle, WA 98122, USA\\[-0.8ex]
\small \texttt{nugentb@seattleu.edu}
}

\maketitle

\begin{abstract}
We show that the $h$-vector of a $2$-dimensional PS ear-decomposable simplicial complex is a pure $\mathcal{O}$-sequence. This provides a strengthening of Stanley's conjecture for matroid $h$-vectors in rank $3$.  Our approach modifies the approach of combinatorial shifting for arbitrary simplicial complexes to the setting of $2$-dimensional PS ear-decomposable complexes, which allows us to greedily construct a corresponding pure multicomplex. 
\end{abstract}

\section{Introduction and Background}

In the late 1970s,  Stanley \cite{Stanley-CM} conjectured that the $h$-vector of a matroid simplicial complex is a pure $\mathcal{O}$-sequence. Later,  Chari \cite{Chari} showed that any matroid simplicial complex admits a PS ear-decomposition, which inductively decomposes a simplicial complex into ears whose contribution to the $h$-vector could correspond to an interval of monomials in the divisibility lattice. As a consequence, it is natural to extend Stanley's conjecture to the family of PS ear-decomposable simplicial complexes.  In this paper, we use this approach to show that $h$-vectors of $2$-dimensional PS ear-decomposable simplicial complexes are pure $\mathcal{O}$-sequences.  We begin with the relevant background on face enumeration for simplicial complexes, then provide further background on Stanley's conjecture and PS ear-decomposable simplicial complexes. 

\subsection{Simplicial complexes and face numbers}

 A \textbf{simplicial complex} $\Delta$ on (finite) vertex set $V = V(\Delta)$ is a collection of subsets $F \subseteq V$ called \textbf{faces} with the property that if $F \in \Delta$ and $G \subseteq F$, then $G \in \Delta$ as well.   To each abstract simplicial complex, there is an associated geometric object called its \textbf{geometric realization}, $\| \Delta \|$, which contains a geometric simplex for each face $F \in \Delta$. 

The \textbf{dimension} of a face $F \in \Delta$ is $\dim(F) = |F|-1$ and the dimension of $\Delta$ is $\dim(\Delta) = \max \{ \dim(F) \ : \ F \in \Delta\}$.  We say that $\Delta$ is \textbf{pure} if all of its facets (maximal faces under inclusion) have the same dimension.  We will typically assume that $\Delta$ is $(d-1)$-dimensional and pure, meaning each of its facets contains $d$ vertices. 

The most fundamental combinatorial data of a $(d-1)$-dimensional simplicial complex $\Delta$ is encoded in its \textbf{$f$-vector}, $f(\Delta) = (f_{-1}(\Delta), f_0(\Delta), \ldots, f_{d-1}(\Delta))$, where the \textbf{$f$-numbers} $f_i(\Delta)$ count the number of $i$-dimensional faces in $\Delta$.  For example, $f_0(\Delta)$, $f_1(\Delta)$, and $f_2(\Delta)$ respectively count the number of vertices, edges, and triangular faces in $\Delta$.  Unless $\Delta$ itself is the empty complex, $f_{-1}(\Delta) = 1$, corresponding to the empty face. 

In many cases, it is more natural to perform a combinatorial transformation the $f$-vector of $\Delta$ to obtain the \textbf{$h$-vector}, $h(\Delta) = (h_0(\Delta), h_1(\Delta), \ldots, h_d(\Delta))$, whose entries are given by $$h_j(\Delta) = \sum_{i=0}^j (-1)^{j-i}\binom{d-i}{d-j}f_{i-1}(\Delta).$$ 

Concretely, when $d=3$ (that is, $\dim(\Delta) = 2$), the $h$-numbers of $\Delta$ are given by 
\begin{eqnarray*}
h_0(\Delta) &=& 1 \\
h_1(\Delta) &=& f_0(\Delta) - 3 \\
h_2(\Delta) &=& f_1(\Delta) - 2f_0(\Delta) + 3 \\
h_3(\Delta) &=& f_2(\Delta) - f_1(\Delta) + f_0(\Delta) - 1.
\end{eqnarray*}

\subsection{PS ear-decomposable complexes}

Chari \cite{Chari} defined the family of $(d-1)$-dimensional PS ear-decomposable simplicial complexes.  Before defining PS ear-decomposable simplicial complexes, we need to define the components that are used to build them, which are PS spheres and PS balls.  

\begin{definition}
A \textbf{PS sphere} is a triangulated sphere that can be decomposed as a join of simplex boundaries.  A \textbf{PS ball} is a triangulated ball that can be decomposed as the join of a simplex and a PS sphere.
\end{definition}

Let us illustrate these definitions when $d=3$.  We will use $\sigma^k$ to denote the $k$-dimensional simplex and $\partial \sigma^k$ to denote its boundary.  For example, when $k=3$, $\sigma^3$ is a solid tetrahedron and $\partial \sigma^3$ is its boundary.  When $k=0$ we set $\partial \sigma^0 = \{\emptyset$\} by convention. 

For $d=3$, the possible PS $2$-spheres and PS $2$-balls are shown in Figure \ref{figure:psballsspheres}. The boundary of each PS ball is highlighted in red. 

\begin{figure}
\begin{center}
\begin{tabular}{|m{1.1in}|m{1.1in}||m{1.1in}|m{1.1in}|} \hline
\multicolumn{2}{|c||}{PS Spheres} & \multicolumn{2}{c|}{PS Balls} \\ \hline
Decomposition & Geometry & Decomposition & Geometry \\ \hline
$\partial \sigma^3$ & 
\begin{center}\begin{tikzpicture}
  \draw[fill=gray!25] (0,0) -- (1.5,0) -- (2.25,1.3) -- (.75,1.3) -- (0,0);
  \draw (.75,1.3) -- (1.5,0);
  \draw[dashed] (0,0) -- (2.25,1.3);
  \foreach \p in {(0,0), (1.5,0), (.75,1.3), (2.25,1.3)}{
  	\draw[fill=black] \p circle (.1);
  } 
    \end{tikzpicture}\end{center}

&
$\sigma^0 * \partial \sigma^2$ & 
\begin{center}
\begin{tikzpicture}
\draw[red, thick , fill=gray!25]  (0,0) -- (2,0) -- (1,1.73) -- (0,0);
\foreach \p in {(0,0), (2,0), (1,1.73)}{
	\draw (1,.5) -- \p;
	\draw[red, fill=red] \p circle (.1); 
}
\draw[fill=black] (1,.5) circle (.1);
\end{tikzpicture}
\end{center}
\\ \hline
$\partial \sigma^1 * \partial \sigma^2$ & 

\begin{center}
\begin{tikzpicture}[scale=1.1]
\def \ts {(.25,1.25,-.25)};
\def \bs{(.25,-1.25,-.25)};
\def \a{(1,0,0)};
\def \b{(-.5,0,0)};
\def \c{(-.05,0,-1.15)};
\draw[fill=gray!25] \ts -- \a -- \b -- \ts;
\draw[fill=gray!25] \bs -- \a -- \b -- \bs;
\draw[fill=black] \a circle (.07);
\draw[fill=black] \b circle (.07);
\draw[fill=black] \c circle (.07);
\draw[fill=black]  \ts circle (.07);
\draw[fill=black] \bs circle (.07);
\draw[dashed] \c -- \ts;
\draw[dashed] \c -- \bs;
\draw \a -- \b;
\draw[dashed] \a -- \c  -- \b;
\end{tikzpicture}

\end{center}
&
$\sigma^0 * \partial \sigma^1 * \partial \sigma^1$ & 
\begin{center}
\begin{tikzpicture}
\draw[red, thick , fill=gray!25]  (0,0) -- (2,0) -- (2,2) -- (0,2) -- (0,0);
\foreach \p in {(0,0), (2,0), (2,2), (0,2)}{
	\draw (1,1) -- \p;
	\draw[red, fill=red] \p circle (.1); 
}
\draw[fill=black] (1,1) circle (.1);
\end{tikzpicture}
\end{center}

\\ \hline
$\partial \sigma^1 * \partial \sigma^1 * \partial \sigma^1$ & 
\begin{center}
\begin{tikzpicture}[scale=1.1]
\draw[fill=gray!25] (1,0,0) -- (0,1.41,0) -- (0,0,.707) -- (1,0,0);
\draw[fill=gray!25] (-1,0,0) -- (0,1.41,0) -- (0,0,.707) -- (-1,0,0);
\draw[fill=gray!25] (1,0,0) -- (0,-1.41,0) -- (0,0,.707) -- (1,0,0);
\draw[fill=gray!25] (-1,0,0) -- (0,-1.41,0) -- (0,0,.707) -- (-1,0,0);
\draw[dashed] (1,0,0) -- (0,0,-.707) -- (-1,0,0);
\draw[dashed] (0,1.41,0) -- (0,0,-.707) -- (0,-1.41,0);
\draw[dashed] (0,1.41,0) -- (0,0,-.707) -- (0,-1.41,0);
\draw[fill=black] (1,0,0) circle (.07);
\draw[fill=black] (-1,0,0) circle (.07);
\draw[fill=black] (0,1.41,0) circle (.07);
\draw[fill=black] (0,-1.41,0) circle (.07);
\draw[fill=black] (0,0,.707) circle (.07);
\draw[fill=black] (0,0,-.707) circle (.07);
\end{tikzpicture}

\end{center}&
$\sigma^1 * \partial \sigma^1$ &
\begin{center}\begin{tikzpicture}
  \draw[red, thick, fill=gray!25] (0,0) -- (1.5,0) -- (2.25,1.3) -- (.75,1.3) -- (0,0);
  \draw (.75,1.3) -- (1.5,0);

  \foreach \p in {(0,0), (1.5,0), (.75,1.3), (2.25,1.3)}{
  	\draw[red, fill=red] \p circle (.1);
  } 
    \end{tikzpicture}\end{center}

\\ \hline
&
&
$\sigma^2$ & 
\begin{center}
\begin{tikzpicture}
\draw[red, thick , fill=gray!25]  (0,0) -- (2,0) -- (1,1.73) -- (0,0);
\foreach \p in {(0,0), (2,0), (1,1.73)}{
	\draw[red, fill=red] \p circle (.1); 
}
\end{tikzpicture}
\end{center}

\\ \hline

\end{tabular}
\end{center}
\caption{PS spheres and balls in dimension 2} 
\label{figure:psballsspheres}
\end{figure}

This leads us to the main object of study in this paper, PS ear-decomposable simplicial complexes. 

\begin{definition}
Let $\Delta$ be a pure $(d-1)$-dimensional simplicial complex.  A \textbf{PS ear-decomposition} of $\Delta$ is a decomposition of $\Delta$ into subcomplexes $\Delta = \Sigma_0 \cup \Sigma_1 \cup \cdots \cup \Sigma_t$ such that
\begin{enumerate}
\item $\Sigma_0$ is a PS $(d-1)$-sphere, 
\item $\Sigma_j$ is a PS $(d-1)$-ball for $1 \leq j \leq t$, and
\item $\displaystyle \Sigma_j \cap \left( \bigcup_{i=0}^{j-1} \Sigma_i \right)= \partial \Sigma_j$ for all $1 \leq j \leq t$.
\end{enumerate}
We say $\Delta$ is \textbf{PS ear-decomposable} if it admits a PS ear-decomposition. 
\end{definition}
In other words, a PS ear-decomposable complex can be constructed inductively by starting with a PS sphere, then repeatedly gluing PS balls to the existing complex such that each successive gluing takes place along the boundary of the corresponding ball.

\subsection{Multicomplexes and $O$-sequences}

The second main objects of study in this paper are multicomplexes, which are multiset analogues of simplicial complexes. A \textbf{multicomplex} $\mathcal{M}$ is a (finite) collection of monomials that is closed under divisibility; i.e., if $\mu \in \mathcal{M}$ and $\nu | \mu$, then $\nu \in \mathcal{M}$. A multicomplex is \textbf{pure} if all of its maximal monomials (under divisibility) have the  same degree (with $\deg(x_i) = 1$ for all $i$).  

Just as the $f$-vector of a simplicial complex $\Delta$ counts the faces in $\Delta$ by cardinality, a multicomplex  has a corresponding \textbf{$F$-vector} that enumerates its elements by degree.  If $\mathcal{M}$ is a multicomplex and $d$ is the maximal degree of a monomial in $\mathcal{M}$, then the $F$-vector is $F(\mathcal{M}) = (F_0, F_1, \ldots, F_d)$, where $F_i(\mathcal{M})$ counts the number of monomials in $\mathcal{M}$ of degree $i$.  

A vector $F = (F_0, F_1, \ldots, F_d) \in \mathbb{Z}_{\geq 0}^{d+1}$ that can be realized as the $F$-vector of a (pure) multicomplex is called a \textbf{(pure) $\mathcal{O}$-sequence}.  

\begin{example}
The vector $F = (1, 3, 5, 3)$ is a pure $\mathcal{O}$-sequence, with corresponding multicomplex $\mathcal{M} = \{1, x, y, z, x^2, xy, xz, y^2, yz, x^3, xyz, y^3\}$.  We see that $\mathcal{M}$ is a multicomplex because the divisors of any monomial of $\mathcal{M}$ also belong to $\mathcal{M}$.  For example, $xyz \in \mathcal{M}$, and its divisors -- $1, x, y, z, xy, xz, yz$ -- also belong to $\mathcal{M}$.  Additionally, $\mathcal{M}$ is pure because every monomial in $\mathcal{M}$ is a divisor of a degree-3 monomial in $\mathcal{M}$.  
\end{example}

\begin{example}
The vector $F = (1, 3, 1)$ is an $\mathcal{O}$-sequence, but not a pure $\mathcal{O}$-sequence.  The monomials $\mathcal{M} = \{1, x, y, z, xy\}$ form a multicomplex with $F$-vector $(1, 3, 1)$.  However, $F$ is not a pure $\mathcal{O}$-sequence because any degree-$2$ monomial has at most two degree-$1$ divisors.  Therefore, a pure $\mathcal{O}$-sequence with $F_2 = 1$ must have $F_1 \leq 2$. 
\end{example}

\begin{example}
The $h$-vector of a PS sphere is a pure multicomplex. For PS 2-spheres we exhibit their $h$-vectors and corresponding pure multicomplexes: 
\begin{itemize}
\item If $\Sigma_0 = \partial \sigma^3$, then $h(\Sigma_0) = (1,1,1,1)$, which corresponds to the pure multicomplex $\mathcal{M} = \{1, x, x^2, x^3\}$. 
\item If $\Sigma_0 = \partial \sigma^1*\partial\sigma^2$, then $h(\Sigma_0) = (1,2,2,1)$, which corresponds to the pure multicomplex $\mathcal{M} = \{1, x, y, x^2, xy, x^2y\}$. 
\item If $\Sigma_0 = \partial \sigma^1*\partial \sigma^1*\partial \sigma^1$, then $h(\Sigma_0) = (1,3,3,1)$, which corresponds to the pure multicomplex $\mathcal{M} = \{1, x, y, z, xy, xz, yz, xyz\}$. 
\end{itemize}
In general, if $\Sigma_0 = \partial \sigma^{d_1} * \partial \sigma^{d_2} * \cdots * \partial \sigma^{d_k}$, its corresponding pure multicomplex contains all divisors of the monomial $x_1^{d_1}x_2^{d_2}\cdots x_k^{d_k}$. 
\end{example}

\subsection{Matroid $h$-vectors and Stanley's conjecture}

Having defined PS ear-decoposable simplicial complexes, multicomplexes, and $\mathcal{O}$-sequences, we are finally in a position to state the problem we wish to study.  

In the late 1970's, Stanley \cite{Stanley-CM} established a deep connection between commutative algebra and the combinatorics of $f$- and $h$-vectors of certain families of simplicial complexes in the following result. 

\begin{theorem}
A vector $h = (h_0, h_1, \ldots, h_d) \in \mathbb{Z}_{\geq 0}^d$ is the $h$-vector of a $(d-1)$-dimensional Cohen-Macaulay simplicial complex if and only if it is an $\mathcal{O}$-sequence. 
\end{theorem}

We will not formally define Cohen-Macaulay simplicial complexes at this point as they will not play a direct role in the rest of this paper, however it is worth noting that triangulations of spheres and balls are Cohen-Macaulay simplicial complexes.  Another family of interesting simplicial complexes is the family of matroid simplicial complexes.  

\begin{definition}
A \textbf{matroid} is a nonempty simplicial complex $\Delta$ that satisfies the following additional property: if $F$ and $G$ are faces in $\Delta$ with $|F| < |G|$, then there exists an element $x \in G \setminus F$ such that $F \cup \{x\}$ is also a face of $\Delta$.
\end{definition}

The extra structure imposed by this so-called exchange axiom adds tremendous structure to matroid simplicial complexes, as is evidenced by the following theorem of Chari. 

\begin{theorem} {\rm{(Chari, \cite[Theorem 3]{Chari})}} \label{thm:matroidspsed}
If $\Delta$ is a coloop-free matroid simplicial complex, then $\Delta$ is PS ear-decomposable.
\end{theorem}

Here, the condition that $\Delta$ is coloop-free means that topologically, $\Delta$ is not a cone.  As such, there is no harm in considering only coloop-free matroids because cone vertices only append zeros to the end of the $h$-vector. Further, the following conjecture of Stanley has remained tantalizingly open for the past several decades: 

\begin{conjecture} {\rm{(Stanley's Conjecture \cite[p. 59]{Stanley-CM}) \\}}
The $h$-vector of a matroid simplicial complex is a pure $\mathcal{O}$-sequence.
\end{conjecture}

It follows from Theorem \ref{thm:matroidspsed} that matroid simplicial complexes are Cohen-Macaulay and hence their $h$-vectors are $\mathcal{O}$-sequences.  The huge difficulty comes in establishing the purity condition. Based on Stanley's Conjecture and Chari's Theorem, the following conjecture is also quite natural. 

\begin{conjecture}
Let $\Delta$ be a PS ear-decomposable simplicial complex.  Then $h(\Delta)$ is a pure $\mathcal{O}$-sequence.
\end{conjecture}

Imani et al. \cite{Imani-et-al} established this result for $1$-dimensional PS ear-decomposable simplicial complexes.  Our goal in this paper is to establish the same result in dimension $2$.  Our main theorem is the following result.

\begin{theorem} \label{thm:mainresult}
Let $\Delta$ be a $2$-dimensional PS ear-decomposable simplicial complex.  Then $h(\Delta)$ is a pure $\mathcal{O}$-sequence. 
\end{theorem}

Before we delve into the technical results leading up to the proof, we will begin with a broad overview.  Let $\Delta = \Sigma_0 \cup \Sigma_1 \cup \cdots \cup \Sigma_t$ be a $2$-dimensional PS ear-decomposable simplicial complex.  We will transform $\Delta$ into a new PS ear-decomposable simplicial complex $\mathcal{C}(\Delta) = \Sigma_0' \cup \Sigma_1' \cup \cdots \cup \Sigma_t'$ with the same number of ears such that $h(\Delta) = h(\mathcal{C}(\Delta))$.  We call $\mathcal{C}(\Delta)$ the \textbf{compression} of $\Delta$, which will serve as an analogue of classical compression/shifting operators to in the setting of PS ear-decomposable simplicial complexes.  

For arbitrary simplicial complexes, compression and shifting operators play an important role in the characterization of $f$-vectors of simplicial complexes in the Kruskal-Katona Theorem \cite{Erdos-Ko-Rado, Stanley-greenbook} and also in Stanley's characterization of $h$-vectors of Cohen-Macaulay simplicial complexes \cite{Stanley-CM}. The difference in this setting is that $\mathcal{C}(\Delta)$ is not generally a shifted or compressed simplicial complex because these operators do not preserve PS ear-decomposability or even purity of the underlying simplicial complex.  Instead, the ears in the PS ear-decomposition of $\mathcal{C}(\Delta)$, are added greedily with respect to revlex order (in a certain sense that will be made more precise later), and to each ear we define a corresponding set of monomials, also chosen greedily, that can be used to explicitly construct a corresponding pure multicomplex whose $F$-vector is $h(\mathcal{C}(\Delta))$. 

The rest of the paper is structured as follows.  In Section \ref{section:shifting-ops} we study the extremal combinatorics of the underlying graph of a $2$-dimensional PS ear-decomposable simplicial complex, proving the main technical results we will use in guaranteeing the existence of the compressed complexes $\mathcal{C}(\Delta)$. In Section \ref{section:main} we prove the main result. Section \ref{section:main} has three main subsections corresponding to the three different PS-spheres of dimension $2$.  The majority of the work goes into handling the case that $\Sigma_0 = \partial \sigma^3$ is the boundary of a tetrahedron.  In the case that $\Sigma_0 = \partial \sigma^1 * \partial \sigma^2$ or $\Sigma_0  = \partial \sigma^1 * \partial \sigma^1 * \partial \sigma^1$ is a bipyramid or octahedron, the proof generally reduces to the case of the tetrahedron, with some exceptions for handling small boundary cases. Ultimately, the proof of Theorem \ref{thm:mainresult} is given in Theorems \ref{thm:basespheretet}, \ref{thm:basespherebipyramid}, and \ref{thm:basesphereoctahedron}.

\section*{Acknowledgments}

We gratefully acknowledge support from NSF grant  DMS-1600048.

\section{Shifting operators and constructible graphs} \label{section:shifting-ops}

If $\Delta$ is a $2$-dimensional PS ear-decomposable simplicial complex with compression $\mathcal{C}(\Delta) = \Sigma_0' \cup \Sigma_1' \cup \cdots \cup \Sigma_t'$, our goal is to assign a family of monomials $\mathcal{M}_i$ to each $\Sigma_i'$ so that $\mathcal{M}_0 \cup \cdots \cup \mathcal{M}_i$ is a pure multicomplex whose $F$-vector is the same as the $h$-vector of $\Sigma_0' \cup \cdots \cup \Sigma_i'$ for each $i$.  As we will see, it is relatively easy to describe the families $\mathcal{M}_i$ when $\Sigma_i$ is one of $\sigma^0 * \partial \sigma^2$, $\sigma^0 * \partial \sigma^1 * \partial \sigma^1$, or $\sigma^1 * \partial \sigma^0$, but it is more difficult when $\Sigma_i = \sigma^2$ fills a missing triangle.  Filling a missing triangle contributes $(0,0,0,1)$ to the $h$-vector, so the challenge in studying missing triangles is to know that there cannot be so many missing triangles to be filled that they would exceed the possible support of degree-2 monomials in the multicomplex. 

Throughout this section, we will use $T(G)$ to denote the set of of triangles ($3$-cycles) in a simple graph $G$ and $\#T(G)$ to denote the cardinality of that set.  In order to better understand triangles in $G$, we begin by exploring the extremal combinatorics of the graph of a PS ear-decomposable simplicial complex.  

Let $G = (V,E)$ be a simple graph with vertices $v_1, v_2, \ldots, v_n$.  For any distinct $i,j \in [n]$, define an operator $\mathcal{S}_{i,j}$ that acts on the edges of $G$ as follows
$$
\mathcal{S}_{i,j}(e) = 
\begin{cases}
(e \setminus \{v_j\}) \cup \{v_i\} & \text{ if } v_j \in e \text{ and } v_i \notin e \text{ and } (e \setminus \{v_j\}) \cup \{v_i\} \notin E \\
e & \text{ otherwise.}
\end{cases}
$$
In other words, $\mathcal{S}_{i,j}$ shifts edges incident to vertex $v_j$ to become incident to vertex $v_i$ whenever possible.  We slightly abuse notation and use $\mathcal{S}_{i,j}(G)$ to denote the resulting graph. 

\begin{lemma} \label{lemma:shifttriangles}
Let $G$ be a simple graph on vertex set $\{v_1,v_2,\ldots,v_n\}$.  For any distinct $i,j \in [n]$, $$\#T(G) \leq \#T(\mathcal{S}_{i,j}(G)).$$
\end{lemma}

\begin{proof}
We establish an injective map from $T(G)$ to $T(\mathcal{S}_{i,j}(G))$. 

Let $\tau = \{v_k,v_\ell,v_m\}$ be a set of vertices that span a triangle in $G$.  If $v_j \notin \tau$, then $\tau$ will remain unaffected by $\mathcal{S}_{i,j}$.  Similarly, if $v_i \in \tau$ and $v_j \in \tau$, then $\tau$ also remains unaffected by $\mathcal{S}_{i,j}$.   In either case, $\tau$ is also a triangle in $\mathcal{S}_{i,j}(G)$.  

Thus it remains to consider the case that $\tau = \{v_j,v_k,v_\ell\}$ with $v_i \notin \tau$. If $\{v_i,v_k\}$ and $\{v_i,v_\ell\}$ are edges in $G$, then once again $\tau$ will remain unaffected by $\mathcal{S}_{i,j}$.  Otherwise, $\{v_i,v_k,v_\ell\}$ is a triangle in $\mathcal{S}_{i,j}(G)$ but not in $G$.

\end{proof}

\subsection{Constructible graphs}

In this section we define a family of graphs called constructible graphs, which arise as graphs of 2-dimensional PS ear-decomposable simplicial complexes.  We will use the following graph theoretical notation: for a vertex $v$ in a graph $G$, the \textbf{degree} of $v$ will be denoted as $\deg(v) = \deg_G(v)$ and $\mathcal{N}(v) = \mathcal{N}_G(v) = \{u \in V(G): \{u,v\} \in E(G)\}$ will denote the \textbf{neighborhood} of $v$.  For $W \subseteq V(G)$, the \textbf{restriction} of $G$ to $W$ is the graph $G|_W$, with vertex set $W$ and edge set $\{\{u,v\} \in E(G) \ : \ u,v \in W\}$.

\begin{definition}
Let $G$ be a simple graph on vertex set $\{v_1,v_2,\ldots,v_n\}$ with a subset of edges labeled by elements of $[n] \cup \{0\}$.  We say $G$ is \textbf{constructible} if one of the following conditions is satisfied: 
\begin{enumerate}
\item $G = K_4$ with all edges labeled $0$, 
\item there exists a vertex $v_\ell \in [n]$ such that $\deg(v_\ell) = 3$ or $\deg(v_\ell) = 4$, all edges incident to vertex $v_\ell$ have label $\ell$, and $G-v_\ell$ is constructible, or
\item there exists an unlabeled edge $v_iv_j \in G$ such that $G-v_iv_j$ is constructible.
\end{enumerate}
\end{definition}

Viewing this recursive definition as an inductive one, constructible graphs are obtained from the complete graph $K_4$ with  edges labeled $0$ through a sequence of three possible operations: adding a new vertex $v_\ell$ of degree three or four, all of whose edges are labeled $\ell$, or inserting an unlabeled missing edge.  Given a constructible graph, one can roughly see the process through which it was constructed (up to reordering) because the edges labeled $\ell > 0$ specify which edges were created at the same time as vertex $v_\ell$ and all other edges were either missing edges that were inserted (unlabeled) or edges that were part of the initial $K_4$ (labeled $0$). 

Constructible graphs are relevant to us because the graph of a $2$-dimensional PS ear-decomposable simplicial complex is constructible. Now we wish to bound the number of triangles in a constructible graph in terms of the types of moves that were used in its construction.  For convenience, we will say that an \textbf{A-move} on a constructible graph  consists of adding a new vertex of degree 3 with appropriately labeled edges, a \textbf{B-move} adds a new vertex of degree 4 with appropriately labeled edges, and an \textbf{E-move} adds an unlabeled missing edge. For a vertex $v \in V(G)$ that is not part of the initial $K_4$, we define its \textbf{type}, $\type(v)$, to equal 3 or 4 depending on the degree of $v$ when it is created. 

Classically, the Kruskal-Katona Theorem tells us that, among all simple graphs with a given number of vertices and edges, the one with the maximal number of triangles is obtained by adding edges reverse lexicographically.  Such graphs are known as \textbf{compressed graphs}.  Of course, this construction may lead to a graph with a large number of isolated vertices, which is not suitable to our definition of constructible graphs.  Therefore, we modify this definition to better suit our needs. 

Let $G$ be a constructible graph, and let $a$, $b$, and $e$ respectively denote the number of A-, B-, and E-moves used in constructing $G$. We define the \textbf{compression} $\mathcal{C}(G)$ to be the  constructible graph on vertex set $\{v_1,v_2,\ldots, v_{4+a+b}\}$ that is built as follows: 

\begin{enumerate}
\item Begin with the complete graph $K_4$ on vertex set $\{v_1,v_2,v_3,v_4\}$ with edges labeled $0$. 
\item For $5 \leq \ell \leq 4+b$, perform a B-move to add vertex $\ell$ and edges $\{v_1,v_\ell\}$, $\{v_2,v_\ell\}$, $\{v_3,v_\ell\},$ $\{v_4,v_\ell\}$, all labeled $\ell$. 
\item For $5+b \leq \ell \leq 4+b+a$, perform an A-move to add vertex $v_\ell$ and edges $\{v_1,v_\ell\}$, $\{v_2,v_\ell\}$, $\{v_3,v_\ell\}$, all labeled $\ell$. 
\item Perform E-moves to insert the $e$ smallest missing edges in reverse lexicographic order. 
\end{enumerate}

Given this definition, we can make the connection to the classical Kruskal-Katona theorem more precise.  For graphs, the Kruskal-Katona theorem says that, among all graphs on $n$ vertices with $m$ edges, the graph on vertex set $\{v_1,\ldots,v_n\}$ whose edges are the first $m$ edges in revlex order has the largest number of triangles. The structure of such a graph can be described simply -- we can uniquely express $m = \binom{p}{2} + q$ with $0 \leq q < p$, and the compressed graph $G$ has the following properties: 
\begin{itemize}
\item $G$ contains a clique on $\{v_1,\ldots,v_p\}$,
\item $G$ contains edges $\{v_i,v_{p+1}\}$ for $1 \leq i \leq q$, and 
\item the vertices $v_{p+2},\ldots,v_n$ are isolated.
\end{itemize} 

Similarly, a compressed constructible graph comes equipped with a vertex order $v_1,\ldots,v_n$ such that 
\begin{itemize}
\item $\type(v_5) \geq \type(v_6) \geq \cdots \geq \type(v_n)$, 
\item the maximal clique spans vertices $\{v_1,\ldots,v_p\}$ for some $p \geq 4$,
\item if $\deg(v_{p+1}) = q$, then $v_{p+1}$ is adjacent to the vertices $v_1,\ldots,v_q$, and
\item $v_i$ is adjacent to the first $\type(v_i)$ vertices for all $i > p+1$. 
\end{itemize}

This connection to the classical Kruskal-Katona Theorem will be made precise in Theorem \ref{thm:compressiontriangles}, which states that $\mathcal{C}(G)$ has  at least as many triangles as $G$ when $G$ is constructible.  The remainder of this section focuses on the proof of Theorem \ref{thm:compressiontriangles}, which requires a few intermediate lemmas.

\begin{lemma} \label{lemma:moveedges}
Let $G$ be a constructible graph.  Assume that $G$ contains a clique on the vertices $\{v_1,\ldots,v_p\}$.  Let $u$ and $v$ be vertices not among $\{v_1,\ldots,v_p\}$ such that 
\begin{itemize}
\item $\mathcal{N}(u) \subseteq \{v_1,\ldots,v_p\}$, 
\item $\mathcal{N}(v) \subseteq \{v_1,\ldots,v_p\}$, 
\item $u$ and $v$ are not adjacent, and
\item $\deg_G(u) \leq \deg_G(v)$.
\end{itemize}
Let $k = \min\{\deg_G(u) - \type(u), p-\deg_G(v)\}$, and let $G'$ be the graph obtained from $G$ by removing $k$ unlabeled edges incident to $u$ and adding $k$ unlabeled edges incident to $v$ whose neighbors lie among $\{v_1,\ldots,v_p\}$.  Then $G'$ is constructible and $\#T(G) \leq \#T(G')$. 
\end{lemma}

\begin{proof}
The number $\deg_G(u) - \type(u)$ is the number of unlabeled edges incident to $u$ and the number $p-\deg_G(v)$ is the number of missing edges between $v$ and vertices among $\{v_1,\ldots,v_p\}$.  So $k$ is the largest number of unlabeled edges that could be moved from $u$ to $v$.  The fact that $G'$ is constructible is immediate as we are only changing the unlabeled edges in $G$. 

Because $\{v_1,\ldots,v_p\}$ span a clique, every pair of vertices in the neighborhood of $u$ (respectively, $v$) form a triangle with $u$ (respectively $v$).  Because of this, and because $u$ and $v$ are not adjacent, we see that
\begin{eqnarray*}
\#T(G') - \#T(G) &=& \binom{\deg_G(u)-k}{2} + \binom{\deg_G(v)+k}{2} - \binom{\deg_G(u)}{2} - \binom{\deg_G(v)}{2} \\
&=& (\deg_G(v)-\deg_G(u))\cdot k+k^2 \geq 0.
\end{eqnarray*}

\end{proof}

\begin{lemma} \label{lemma:neighborhoodbound}
Let $G$ be a constructible graph, and let $W \subseteq V(G)$ be a subset of vertices.  Then $$\#E\left(G|_W\right) \leq \#E\left(\mathcal{C}(G)|_{\{v_1,\ldots,v_{|W|}\}}\right).$$
\end{lemma}

\begin{proof}
We prove the claim by induction on the number of edges in $G|_W$.  The claim is trivial when there are no edges. 

First suppose that there exists an unlabeled edge $e \in G|_W$. By the inductive hypothesis, 
$$\#E\left((G-e)|_W\right) \leq \#E\left(\mathcal{C}(G-e)|_{\{v_1,\ldots,v_{|W|}\}}\right).$$  
To obtain $\mathcal{C}(G)$ from $\mathcal{C}(G-e)$, we add the revlex smallest missing edge. If the revlex smallest missing edge is $\{v_i,v_j\}$ and $j \leq |W|$, then 
$$\#E(G|_W) = \#E\left((G-e)|_W\right) + 1 \leq \#E\left(\mathcal{C}(G-e)|_{\{v_1,\ldots,v_{|W|}\}}\right) + 1 = \#E\left(\mathcal{C}(G)|_{\{v_1,\ldots,v_{|W|}\}}\right).$$
Otherwise, if the revlex smallest missing edge is $\{v_i,v_j\}$ and $j > |W|$, then $\mathcal{C}(G)|_{\{v_1,\ldots,v_{|W|}\}}$ is complete, in which case 
$$\#E\left(G|_W\right) \leq \#E\left(\mathcal{C}(G)|_{\{v_1,\ldots,v_{|W|}\}}\right)$$
holds trivially. 

Thus, we are left to handle the case that $G|_W$ contains only labeled edges. Order the vertices in $W$ as $v_{i_1},v_{i_2},\ldots,v_{i_{|W|}}$ in order of their creation.  In $G$, there is at most one edge whose largest vertex is $v_{i_2}$, at most two edges whose largest vertex is $v_{i_3}$, at most three edges whose largest vertex is $v_{i_4}$, and at most $\type(v_{i_j})$ edges whose largest vertex is $v_{i_j}$ for $j > 4$.  The same is true in $\mathcal{C}(G)$, and because the vertices of $\mathcal{C}(G)$ are ordered so that vertices of type $4$ come before vertices of type $3$, it must be the case that 
$$\#E\left(G|_W\right) \leq \#E\left(\mathcal{C}(G)|_{\{v_1,\ldots,v_{|W|}\}}\right).$$
\end{proof}

\begin{theorem} \label{thm:compressiontriangles}
Let $G$ be a constructible graph.  Then $$\#T(G) \leq \#T(\mathcal{C}(G)).$$
\end{theorem}

\begin{proof}
We prove the claim by induction on the number of vertices in $G$.  When $\#V(G) = 4$, the claim is clear as $G$ and $\mathcal{C}(G)$ are both the complete graph.  Therefore, we may assume $\#V(G) = n+1 > 4$. 

If possible, pick a vertex $v \in V(G)$ with $\type(v) = 4$.  Otherwise, all vertices in $G$ that are not part of the original $K_4$ have type 3; pick one of those vertices arbitrarily. Let $\delta = \deg_G(v)$.  Note that the number of triangles in $G$ that contain $v$ is equal to the number of edges in $G|_{\mathcal{N}(v)}$.  Therefore, by the inductive hypothesis and Lemma \ref{lemma:neighborhoodbound}, 
\begin{eqnarray}
\#T(G) &=& \#T(G-v) + \#E(G|_{\mathcal{N}(v)}) \nonumber \\
\label{eq1}
&\leq& \#T(\mathcal{C}(G-v)) + \#E\left(\mathcal{C}(G-v)|_{\{v_1,\ldots,v_\delta\}}\right).
\end{eqnarray}

Let $G'$ be the graph obtained from $\mathcal{C}(G-v)$ by adding vertex $v$, along with edges $\{v,v_i\}$ for $1 \leq i \leq \delta$, the first $\type(v)$ of which receive label $v$.  This label will be temporary, as we need to determine the position where $v$ should be inserted into the given order on the vertices of $\mathcal{C}(G-v)$. By construction, 
\begin{equation} \label{eq2}
\#T(G') = \#T(\mathcal{C}(G-v)) + \#E\left(\mathcal{C}(G-v)|_{\{v_1,\ldots,v_d\}}\right).
\end{equation}

What is the structure of $\mathcal{C}(G-v)$? Its vertices are ordered $v_1,v_2,\ldots,v_{n}$ in such a way that $\type(v_5) \geq \type(v_6) \geq \cdots \geq \type(v_n)$.  Moreover, there exists an integer $4 \leq p \leq n$ such that
\begin{itemize}
\item  $\{v_1,\ldots,v_p\}$ span a clique, 
\item $\deg(v_{p+1}) < p$ (meaning $\{v_1,\ldots,v_{p+1}\}$ do not span a clique), 
\item $\mathcal{N}(v_i) \subseteq \{v_1,\ldots,v_p\}$ for all $i > p$, and
\item $\deg(v_i) = \type(v_i)$ for all $i > p+1$. 
\end{itemize}
Moreover, for all $i > 4$, the labeled edges incident to $v_i$ are the those incident to the first $\type(v_i)$ vertices --- either $\{v_1,v_2,v_3\}$ or $\{v_1,v_2,v_3,v_4\}$. 

By Eqs \eqref{eq1} and \eqref{eq2}, it follows that $\#T(G) \leq \#T(G')$, so we need only show that $\#T(G') \leq \#T(\mathcal{C}(G))$. We examine several cases. 

\ \\  \textit{Case 1:} $\deg(v) = p$ \\

In this case, $\mathcal{C}(G)$ can be obtained from $G'$ in two steps: 

First, insert vertex $v$ into the given vertex order so that $v_4 < v < v_5$.  Second, if necessary, replace the revlex largest unlabeled edge $\{v_k,v_{p+1}\}$ with the unlabeled edge $\{v,v_{p+1}\}$.  Whether $\type(v) = 4$ or all vertices have type 3, this order respects the rule that all vertices of type 4 come before the vertices of type 3.  Moreover, $\mathcal{C}(G)$ will contain a clique of size $p+1$ on $\{v,v_5,\ldots,v_p\}$, with unlabeled edges $\{v,v_{p+1}\}, \{v_5,v_{p+1}\}, \ldots, \{v_{k-1},v_{p+1}\}$ being the remaining revlex smallest edges under this new vertex order. Neither of these operations changes the number of triangles in the graph, and hence $\#T(G') = \#T(\mathcal{C}(G))$.

\ \\  \textit{Case 2:} $\deg(v) > p$ \\

As above, let $\delta = \deg(v)$ and  insert $v$ into the given vertex order so that $v_4 < v < v_5$. Having done this, the graph $G'$ fails to be compressed because there are unlabeled edges of the form $\{v,v_i\}$ for $p+2 \leq i \leq \delta$ that are not the revlex smallest edges that could have been added. However, it is still the case that $\deg_{G'}(v_{p+1}) \geq \deg_{G'}(v_{p+2}) \geq \cdots \geq \deg_{G'}(v_\delta)$.  Therefore, we may repeatedly apply Lemma \ref{lemma:moveedges}, first to the pair of vertices $\{v_{p+1},v_\delta\}$, then $\{v_{p+1},v_{\delta-1}\}$, and so on, each time moving the unlabeled edge $\{v,v_k\}$ for $k > p+1$ to the revlex smallest missing edge incident to $v_{p+1}$.  Each of these moves weakly increases the number of triangles in the graph, and hence $\#T(G') \leq \#T(\mathcal{C}(G))$ as desired.  We note that at some point in this process, it could be the case that $v_{p+1}$ becomes adjacent to all vertices among $\{v,v_5,\ldots,v_p\}$.  In this case, we simply start moving unlabeled edges so that they form the revlex smallest missing edge incident to $v_{p+2}$ instead (and, if necessary, then to $v_{p+3}$, $v_{p+4}$, etc.).  

\ \\  \textit{Case 3:} $\deg(v) < p$ \\

This is the most complicated of the cases and requires a few sub-cases. 

If $\type(v) = \type(v_{p+1})$, we insert $v$ into the vertex order on $\mathcal{C}(G-v)$ as follows: if $\deg(v) \geq \deg(v_{p+1})$, then $v_p < v < v_{p+1}$; otherwise, if $\deg(v_{p+1}) > \deg(v)$, then  $v_{p+1} < v < v_{p+2}$ if .  Then we apply Lemma \ref{lemma:moveedges} to the pair of vertices $\{v,v_{p+1}\}$ to obtain $\mathcal{C}(G)$ and see that $\#T(G') \leq \#T(\mathcal{C}(G))$.

If $\type(v) \neq \type(v_{p+1})$, then by our choice of $v$ it must be the case that $\type(v) = 4$ and $\type(v_{p+1}) = 3$.  If there exists an index $5 \leq i \leq p$ such that $\type(v_i) = 3$, we can pick the smallest such $i$ and perform an operation that declares $\type(v_i) = 4$ and $\type(v) = 3$.  Since $4 < i \leq p$, we know that $\{v_4,v_i\} \in G'$, so we can give that edge the label $i$ and remove the label $v$ from the edge $\{v_4,v\}$.  This does not change the number of trees in $G'$, but now we have arranged for $\type(v) = 3 = \type(v_{p+1})$ and we can apply the argument used in the previous case.

Thus we need only examine the case that $\type(v_p) = \type(v) = 4$, but $\type(v_{p+1}) = 3$. If $\deg(v) \geq \deg(v_{p+1})$, we insert $v$ into the vertex order on $\mathcal{C}(G-v)$ so that $v_p < v < v_{p+1}$.  This preserves the condition that all vertices of type 4 come before the vertices of type 3 in the vertex order, and then we can apply Lemma \ref{lemma:moveedges} to the pair of vertices $\{v,v_{p+1}\}$ to obtain $\mathcal{C}(G)$ from $G'$.  On the other hand, if $\deg(v) < \deg(v_{p+1})$, we first perform the above operation to swap edge labels so that $\type(v) = 3$ and $\type(v_{p+1}) = 4$.  Then we insert $v$ into the vertex order on $\mathcal{C}(G-v)$ so that $v_{p+1} < v < v_{p+2}$ and then apply Lemma \ref{lemma:moveedges}.

\end{proof}

\section{Proofs of the main result} \label{section:main}

Let $\Delta = \Sigma_0 \cup \Sigma_1 \cup \cdots \cup \Sigma_t$ be a $2$-dimensional PS ear-decomposable simplicial complex.  By a slight abuse of notation from the previous section, we will say the PS balls that are added by $\Sigma_i = \sigma^0 * \partial \sigma^2$, $\Sigma_i = \sigma^0 * \partial \sigma^1 * \partial \sigma^1$, $\Sigma_i = \sigma^1 * \partial \sigma^1$, or $\Sigma_i = \sigma^2$, are called ears of type  A, B, E, and F respectively. Similarly, we let $\eta_A$, $\eta_B$, $\eta_E$, and $\eta_F$ respectively denote the number of ears of each type used in constructing $\Delta$.  This implies that 

\begin{eqnarray}
h_0(\Delta) &=& 1 \\
h_1(\Delta) &=& h_1(\Sigma_0) + \eta_A + \eta_B \\
h_2(\Delta) &=& h_2(\Sigma_0) + \eta_A + 2\eta_B + \eta_E \\
h_3(\Delta) &=& 1 + \eta_A + \eta_B + \eta_E + \eta_F.
\end{eqnarray}

In other words, a move of type A contributes $(0,1,1,1)$ to the $h$-vector, a move of type B contributes $(0,1,2,1)$ to the $h$-vector, a move of type E contributes $(0,0,1,1)$, and a move of type F contributes $(0,0,0,1)$ to the $h$-vector. 

Our goal in this section is to prove that the $h$-vector of a PS ear-decomposable simplicial complex is a pure $\mathcal{O}$-sequence.  The proof breaks into three main cases for the three possible base spheres $\Sigma_0$, along with several subcases.  We begin with a broad overview of the main cases and subcases in the the remainder of the paper.

\begin{enumerate}
\item The case that $\Sigma_0 = \partial \sigma^3$ is the boundary of a tetrahedron is the simplest case given Theorem \ref{thm:compressiontriangles}.  Algorithm \ref{tetrabasealg} defines the compressed complex $\mathcal{C}(\Delta)$ and corresponding pure multicomplex $\mathcal{M}(\Delta)$.  The fact that this algorithm terminates and produces a pure multicomplex is proved in Theorem \ref{thm:basespheretet}. 
\item The case that $\Sigma_0 = \partial \sigma^1 * \partial \sigma^2$ is the boundary of a bipyramid requires two subcases.  
\begin{enumerate}
\item When $\eta_F > 0$, we can reduce the problem to the case that $\Sigma_0 = \partial \sigma^3$ in Theorem \ref{thm:basespherebipyramid-F>0}. 
\item When $\eta_F = 0$, Algorithm \ref{bipyramidbasealg} defines the compressed complex $\mathcal{C}(\Delta)$ and corresponding pure multicomplex $\mathcal{M}(\Delta)$.  The fact that this algorithm terminates and produces a pure multicomplex is proved in Theorem \ref{thm:basespherebipyramid-F=0}.
\end{enumerate}
\item The case that $\Sigma_0 = \partial \sigma^1 * \partial \sigma^1 * \partial \sigma^1$ is the boundary of an octahedron is the most complicated case. 
\begin{enumerate}
\item When $\eta_F = 0$, Algorithm \ref{alg:octahedron-F=0} defines the compressed complex $\mathcal{C}(\Delta)$ and corresponding pure multicomplex $\mathcal{M}(\Delta)$.  The fact that this algorithm terminates and produces a pure multicomplex is proved in Theorem \ref{thm:octahedron-F=0}.
\item When $\eta_E = 0$, Algorithm \ref{alg:octahedron-E=0} defines the compressed complex $\mathcal{C}(\Delta)$ and corresponding pure multicomplex $\mathcal{M}(\Delta)$.  The fact that this algorithm terminates and produces a pure multicomplex is proved in Theorem \ref{thm:octahedron-E=0}.
\item When $\eta_F > 0$ and $\eta_E > 0$, we must consider two further subcases based on whether $\{v_1,v_4\} \in \Delta$.  
\begin{enumerate}
\item When $\{v_1,v_4\} \in \Delta$, we reduce the problem to the case that $\Sigma_0 = \partial \sigma^1 * \partial \sigma^2$ is a bipyramid in Proposition \ref{prop:base-oct-withv1v4}. 
\item When $\{v_1,v_4\} \notin \Delta$, we use shifting operators $\mathcal{S}_{i,j}$ to reduce to the case that $\Sigma_0 = \partial \sigma^3$ is a tetrahedron, but with two extra monomials of degree-3 in Theorem \ref{octahedron-to-tetrahedron}.  We then show that two monomials of degree-3 can be removed in Theorem \ref{octahedron-remove-extra-monomials}, which completes the proof. 
\end{enumerate}
 
\end{enumerate}
\end{enumerate}

Throughout the remainder of the paper, we will label the vertices of the tetrahedron, bipyramid, and octahedron as shown in Figure \ref{fig:base-sphere-labels}.

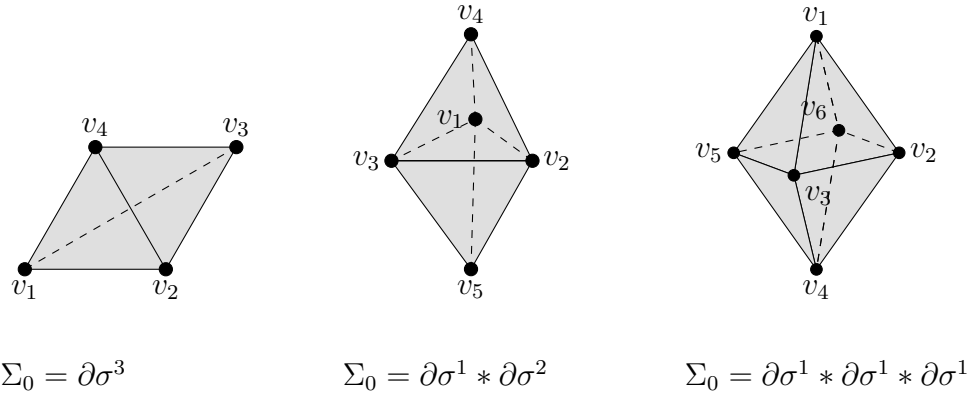
\begin{figure}[h]
\begin{center}
\begin{tabular}{p{.25\textwidth}p{.25\textwidth}p{.25\textwidth}}

\begin{tikzpicture}[scale=1.25]
  \draw[fill=gray!25] (0,0) node[anchor=north] {$v_1$} -- (1.5,0) node[anchor=north] {$v_2$} -- (2.25,1.3) node[anchor=south] {$v_3$} -- (.75,1.3) node[anchor=south] {$v_4$}-- (0,0) ;
  \draw (.75,1.3) -- (1.5,0);
  \draw[dashed] (0,0) -- (2.25,1.3);
  \foreach \p in {(0,0), (1.5,0), (.75,1.3), (2.25,1.3)}{
  	\draw[fill=black] \p circle (.07);
  } 
  
    \end{tikzpicture}

&

\begin{tikzpicture}[scale=1.25]
\def \ts {(.25,1.25,-.25)};
\def \bs{(.25,-1.25,-.25)};
\def \a{(1,0,0)};
\def \b{(-.5,0,0)};
\def \c{(-.05,0,-1.15)};
\draw[fill=gray!25] \ts node[anchor=south] {$v_4$}-- \a node[anchor=west] {$v_2$} -- \b node[anchor=east] {$v_3$} -- \ts;
\draw[fill=gray!25] \bs node[anchor=north]{$v_5$} -- \a -- \b -- \bs;
\draw[fill=black] \a circle (.07);
\draw[fill=black] \b circle (.07);
\draw[fill=black] \c circle (.07) node[anchor=east] {$v_1$};
\draw[fill=black]  \ts circle (.07);
\draw[fill=black] \bs circle (.07);
\draw[dashed] \c -- \ts;
\draw[dashed] \c -- \bs;
\draw \a -- \b;
\draw[dashed] \a -- \c  -- \b;
\end{tikzpicture}
&
\begin{tikzpicture}[scale=1.1]
\draw[fill=gray!25] (1,0,0) -- (0,1.41,0) -- (0,0,.707) -- (1,0,0);
\draw[fill=gray!25] (-1,0,0) -- (0,1.41,0) -- (0,0,.707) -- (-1,0,0);
\draw[fill=gray!25] (1,0,0) -- (0,-1.41,0) -- (0,0,.707) -- (1,0,0);
\draw[fill=gray!25] (-1,0,0) -- (0,-1.41,0) -- (0,0,.707) -- (-1,0,0);
\draw[dashed] (1,0,0) -- (0,0,-.707) -- (-1,0,0);
\draw[dashed] (0,1.41,0) -- (0,0,-.707) -- (0,-1.41,0);
\draw[dashed] (0,1.41,0) -- (0,0,-.707) -- (0,-1.41,0);
\draw[fill=black] (1,0,0) circle (.07) node[anchor=west]{$v_2$};
\draw[fill=black] (-1,0,0) circle (.07) node[anchor=east]{$v_5$};
\draw[fill=black] (0,1.41,0) circle (.07) node[anchor=south]{$v_1$};
\draw[fill=black] (0,-1.41,0) circle (.07) node[anchor=north]{$v_4$};
\draw[fill=black] (0,0,.707) circle (.07) node[anchor=north west]{$v_3$};
\draw[fill=black] (0,0,-.707) circle (.07) node[anchor=south east]{$v_6$};
\end{tikzpicture}
\\
$\Sigma_0 = \partial \sigma^3$ & $\Sigma_0 = \partial\sigma^1 * \partial\sigma^2$ & $\Sigma_0 = \partial \sigma^1 * \partial \sigma^1 * \partial \sigma^1$ 
\end{tabular}
\caption{Vertex labelings for $2$-dimensional PS spheres.}
\label{fig:base-sphere-labels}
\end{center}
\end{figure}

\subsection{The case that $\Sigma_0 = \partial \sigma^3$}

We now present the main algorithm that will be used to prove that $h$-vectors of $2$-dimensional PS ear-decomposable complexes are pure $\mathcal{O}$-sequences when $\Sigma_0 = \partial\sigma^3$. 

\begin{algorithm} \label{tetrabasealg}
Let $\Delta$ be a $2$-dimensional PS ear-decomposable simplicial complex with $\Sigma_0 = \partial \sigma^3$.  We define a new PS ear-decomposable complex $\mathcal{C}(\Delta)$ and a set of monomials $\mathcal{M}(\Delta)$ in the variables $\{x_4,\ldots,x_{\eta_A+\eta_B+4}\}$ inductively as follows. 

\begin{itemize}
\item[Step 0:]
\item Set $\Sigma_0 = \partial \sigma^3$ on vertex set $\{v_1,v_2,v_3,v_4\}$.
\item Set $\mathcal{M}(\Delta) = \{1, x_4,x_4^2,x_4^3\}$. 
\item[Step 1:]
\item For $5 \leq i \leq \eta_B+4$:
\begin{itemize}
\item Introduce vertex $v_i$ to $\mathcal{C}(\Delta)$ through a B-move whose boundary is glued along the cycle $v_1v_2v_3v_4$.

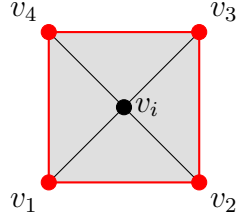
\begin{center}
\begin{tikzpicture}
\draw[red, thick , fill=gray!25]  (0,0) -- (2,0) -- (2,2) -- (0,2) -- (0,0);
\foreach \p in {(0,0), (2,0), (2,2), (0,2)}{
	\draw (1,1) -- \p;
	\draw[red, fill=red] \p circle (.1); 
}
\draw (1,1) node[anchor=west] {$v_i$};
\draw (0,0) node[anchor=north east] {$v_1$};
\draw (2,0) node[anchor=north west] {$v_2$};
\draw (2,2) node[anchor=south west] {$v_3$};
\draw (0,2) node[anchor=south east] {$v_4$};
\draw[fill=black] (1,1) circle (.1);
\end{tikzpicture}
\end{center}
\item Add $\{x_i,x_4x_i,x_i^2,x_4x_i^2\}$ to $\mathcal{M}(\Delta)$.
\end{itemize}

\item For $\eta_B+5 \leq i \leq \eta_A + \eta_B+4$:
\begin{itemize}
\item Introduce vertex $v_i$ to $\mathcal{C}(\Delta)$ through an A-move whose boundary is glued along the cycle $v_1v_2v_3$, 

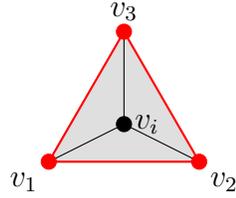
\begin{center}
\begin{tikzpicture}
\draw[red, thick , fill=gray!25]  (0,0) -- (2,0) -- (1,1.73) -- (0,0);
\foreach \p in {(0,0), (2,0), (1,1.73)}{
	\draw (1,.5) -- \p;
	\draw[red, fill=red] \p circle (.1); 
}
\draw[fill=black] (1,.5) circle (.1) node[anchor=west] {$v_i$};
\draw (0,0) node[anchor=north east] {$v_1$};
\draw (2,0) node[anchor=north west] {$v_2$};
\draw (1,1.73) node[anchor=south] {$v_3$};
\end{tikzpicture}
\end{center}
\item Add $\{x_i,x_i^2,x_i^3\}$to $\mathcal{M}(\Delta)$.
\end{itemize}
\item[Step 2:]
\item For $1 \leq \ell \leq \eta_E$: 
\begin{itemize}
\item Let $\{v_i,v_j\}$ be the revlex smallest missing edge in $\mathcal{C}(\Delta)$.
\item Add the edge $\{v_i,v_j\}$ to $\mathcal{C}(\Delta)$ through an E-move whose boundary is the cycle $v_1v_iv_2v_j$.
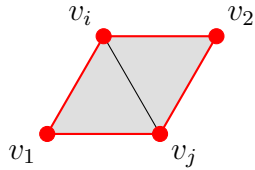
\begin{center}\begin{tikzpicture}
  \draw[red, thick, fill=gray!25] (0,0) -- (1.5,0) -- (2.25,1.3) -- (.75,1.3) -- (0,0);
  \draw (.75,1.3) -- (1.5,0);

  \foreach \p in {(0,0), (1.5,0), (.75,1.3), (2.25,1.3)}{
  	\draw[red, fill=red] \p circle (.1);
  } 
  
\draw (0,0) node[anchor=north east] {$v_1$};
\draw (1.5,0) node[anchor=north west] {$v_j$};
\draw (2.25,1.3) node[anchor=south west] {$v_2$};
\draw (.75,1.3) node[anchor=south east] {$v_i$};

    \end{tikzpicture}\end{center}
\item Add $\{x_ix_j,x_i^2x_j\}$ to $\mathcal{M}(\Delta)$.
\end{itemize}
\item[Step 3:] 
\item For $1 \leq \ell \leq \eta_F$:
\begin{itemize}
\item Let $G$ be the revlex smallest missing $2$-face in $\mathcal{C}(\Delta)$.  Add $G$ to $\mathcal{C}(\Delta)$. 
\item Let $\mu$ be the revlex smallest degree-3 monomial not belonging to $\mathcal{M}(\Delta)$ whose proper divisors all belong to $\mathcal{M}(\Delta)$.  Add $\mu$ to $\mathcal{M}(\Delta)$.
\end{itemize}
\end{itemize}
\end{algorithm}

Note first that upon the completion of Step 1, $\mathcal{M}(\Delta)$ is a pure multicomplex containing $x_j^2$ for all $4 \leq j \leq \eta_A + \eta_B + 4$.  The degree-$2$ monomials that do not belong to $\mathcal{M}(\Delta)$ at this point are those of the form $x_4x_j$ for which vertex $v_j$ was introduced through an A-move, along with all of those of the form $x_ix_j$ with $5 \leq i < j \leq \eta_A + \eta_B + 4$.  Thus the number of such monomials is $\eta_A + \binom{\eta_A+\eta_B}{2}$.  But this is exactly the same as the number of missing edges in $\mathcal{C}(\Delta)$ after the completion of Step 1, meaning that Step 2 of the Algorithm will terminate and the resulting multicomplex $\mathcal{M}(\Delta)$ is still a pure multicomplex. 

Now we turn our attention to Step 3, which is slightly more complicated.  Again we count the number of missing $2$-faces in $\mathcal{C}(\Delta)$ and the number of degree-3 monomials not belonging to $\mathcal{M}(\Delta)$ upon the completion of Step 2. 

Let $G$ be the underlying graph of $\Delta$ and $G'$ the underlying graph of $\mathcal{C}(\Delta)$.  By the construction of $\mathcal{C}(\Delta)$ in Steps 0, 1, and 2, $G' = \mathcal{C}(G)$ is a compressed constructible graph.  Let $p$ be the size of the largest clique in $G'$, and let $q$ denote the degree of $v_{p+1}$. Furthermore, let $a'$ be the number of vertices among $\{v_5,\ldots,v_{p+1}\}$ that were introduced through a type-A move and let $a''$ be the number of vertices among $\{v_{p+2},\ldots,v_n\}$ that were introduced through a type-A move.  Define $b'$ and $b''$ similarly for vertices introduced through a type-B move. 

Observe that $p+1 = 4+a'+b'$, so
\begin{equation}\label{pcount}
p = a'+b'+3.
\end{equation} 
Further,
\begin{equation} \label{etaEcount}
\eta_E = \binom{p}{2} + q - 3a'-4b'-6,
\end{equation}
because there are $\binom{p}{2} + q$ edges among the first $p+1$ vertices of $G'$, but $3a'+4b'$ were introduced as part of the A- or B-moves used to create vertices and 6 were part of the initial PS sphere $\Sigma_0$. 

Next, we can directly count that 
\begin{equation}\label{Tcount}
\#T(G') = \binom{p}{3} + \binom{q}{2} + 6b'' + 3a''.
\end{equation}
However, $3\eta_A + 4\eta_B + 2\eta_E + 4$ of those triangles span triangular faces in $\mathcal{C}(\Delta)$ after the completion of Step 2 because A-, B-, and E-moves respectively introduce $3$, $4$, and $2$ triangular faces, and there are $4$ triangular faces in $\Sigma_0$.  Finally, by Theorem \ref{thm:compressiontriangles}, the graph of $\Delta$ cannot have more triangles than $G'$, and $\Delta$ also has $3\eta_A + 4\eta_B + 2\eta_E$ triangular faces that are introduced through A-, B-, and E-moves, and $4$ additional triangular faces in $\Sigma_0$.  Therefore, 

\begin{eqnarray*}
\eta_F &\leq& \#T(G) - 3\eta_A - 4\eta_B - 2\eta_E - 4 \\
&\leq& \#T(G') - 3\eta_A - 4\eta_B - 2\eta_E - 4 \\
&\stackrel{\eqref{Tcount}}{=}& \binom{p}{3} + \binom{q}{2} + 6b''+3a'' - 3(a'+a'') -4(b'+b'') - 2\eta_E - 4 \\
&\stackrel{\eqref{etaEcount}}{=}& \left[\binom{p}{3}-2\binom{p}{2}\right] + \left[\binom{q}{2}-2q\right] +3a'+4b'+2b''+8 \\
&=& \left[\binom{p-2}{3}-3p+4\right] + \left[\binom{q-2}{2}-3\right]+3a'+4b'+2b''+8 \\
&\stackrel{\eqref{pcount}}{=}& \binom{p-2}{3} + \binom{q-2}{2} + b' + 2b''.
\end{eqnarray*}

On the other hand, upon the completion of Step 2, the degree-$2$ monomials in $\mathcal{M}(\Delta)$ support the following degree-$3$ monomials: 
\begin{itemize}
\item the $\binom{p-1}{3}$ monomials of degree $3$ in the variables $x_4,\ldots,x_p$,
\item the $\binom{q-1}{2}$ monomials of degree $3$ in the variables $x_4,\ldots,x_q,x_{p+1}$ that are divisible by $x_{p+1}$,
\item the $a''$ monomials $x_j^3$ corresponding to vertices $v_j$ that were introduced through an A-move for $j > p+1$, and
\item the $3b''$ monomials of the form $\{x_4^2x_j,x_4x_j^2, x_j^3\}$ corresponding to vertices $v_j$ that were introduced through a B-move for $j>p+1$.
\end{itemize}
Use $m = \binom{p-1}{3} + \binom{q-1}{2} + a''+3b''$ denote the number of available monomials. Upon the completion of Step 2, $1+\eta_A+\eta_B+\eta_E$ of these monomials have been added to $\mathcal{M}(\Delta)$, and hence the number of available degree-$3$ monomials that can be added to $\mathcal{M}(\Delta)$ is

\begin{eqnarray*}
m-\eta_A-\eta_B-\eta_E-1 &=& \binom{p-1}{3} + \binom{q-1}{2} + a''+3b''-\eta_A-\eta_B-\eta_E-1 \\
&\stackrel{\eqref{etaEcount}}{=}& \left[\binom{p-1}{3}-\binom{p}{2}\right] + \left[\binom{q-1}{2}-q\right] \\
&& \qquad+ a''+3b''+3a'+4b'-\eta_A-\eta_B+5 \\
&=& \left[\binom{p-2}{3} -2p+3\right] + \left[\binom{q-2}{2}-2\right] \\
&& \qquad+  a''+3b''+3a'+4b'-\eta_A-\eta_B+5 \\
&\stackrel{\eqref{pcount}}{=}& \binom{p-2}{2} + \binom{q-2}{2} + b'+2b''.
\end{eqnarray*}

This tells us that $\eta_F$ is bounded by the number of degree-$3$ monomials that can be added to $\mathcal{M}(\Delta)$ upon the completion of Step 2 while still preserving the property that $\mathcal{M}(\Delta)$ is a multicomplex.  Consequently, Step 3 will terminate and upon its completion, $\mathcal{M}(\Delta)$ will be a pure multicomplex. This proves the following theorem. 

\begin{theorem} \label{thm:basespheretet}
The set of monomials $\mathcal{M}(\Delta)$ output by Algorithm \ref{tetrabasealg} is a pure multicomplex.  Moreover $F(\mathcal{M}(\Delta)) = h(\mathcal{C}(\Delta)) = h(\Delta)$.  Consequently, if $\Delta$ is a 2-dimensional PS ear-decomposable simplicial complex with $\Sigma_0 = \partial \sigma^3$, then $h(\Delta)$ is a pure $\mathcal{O}$-sequence.
\end{theorem}

\subsection{The case that $\Sigma_0 = \partial \sigma^1 * \partial\sigma^2$}

The case that $\Sigma_0$ is a bipyramid is very similar to the case that $\Sigma_0$ is the boundary of a tetrahedron.  We can start by labeling the vertices of the bipyramid as in Figure \ref{fig:base-sphere-labels}.

Now we examine two cases.  First suppose $\eta_F > 0$.  If at some point the missing face $\{v_1,v_2,v_3\}$ is filled through an F-move, then we can view $\Delta$ as a PS ear-decomposable complex with $\Sigma_0 = \partial \sigma^3$ by starting with the boundary of the tetrahedron on $\{v_1,v_2,v_3,v_4\}$, then adding vertex $v_5$ through an A-move, and then attaching the remaining ears in order to construct $\Delta$.  On the other hand, if $\{v_1,v_2,v_3\}$ is not filled through an F-move, pick any $2$-face filled through an F-move, remove it from $\Delta$, and fill $\{v_1,v_2,v_3\}$ instead.  This creates a new PS ear-decomposable complex $\Delta'$ with $h(\Delta') = h(\Delta)$.  In either of these cases, we can then apply Theorem \ref{thm:basespheretet} to see that $h(\Delta)$ is a pure $\mathcal{O}$-sequence.  This proves the following theorem. 

\begin{theorem} \label{thm:basespherebipyramid-F>0}
Let $\Delta = \Sigma_0 \cup \Sigma_1 \cup \cdots \cup \Sigma_t$ be a 2-dimensional PS ear-decomposable simplicial complex with $\Sigma_0 = \partial \sigma^1 * \partial \sigma^2$ and $\eta_F > 0$.  There exists a PS ear-decomposable simplicial complex $\Delta' = \Sigma_0' \cup \Sigma_1' \cup \cdots \cup \Sigma_t'$ with $\Sigma_0' = \partial \sigma^3$ and $h(\Delta) = h(\Delta')$.  Consequently, $h(\Delta)$ is a pure $\mathcal{O}$-sequence.
\end{theorem}

On the other hand, if $\eta_F = 0$, we can use the following algorithm to construct a pure multicomplex whose $F$-vector is the same as $h(\Delta)$. 

\begin{algorithm}\label{bipyramidbasealg}
Let $\Delta$ be a $2$-dimensional PS ear-decomposable simplicial complex with $\Sigma_0 = \partial \sigma^1 * \partial \sigma^2$ and $\eta_F = 0$. We define a new PS ear-decomposable complex $\mathcal{C}(\Delta)$ and a set of monomials $\mathcal{M}(\Delta)$ in the variables $\{x_4,x_5,\ldots,x_{\eta_A+\eta_B+5}\}$ inductively as follows. 

\begin{itemize}
\item[Step 0:]
\item Set $\Sigma_0 = \partial \sigma^1*\partial\sigma^2$ on vertex set $\{v_1,v_2,v_3,v_4,v_5\}$ as shown in Figure \ref{fig:base-sphere-labels}.
\item Set $\mathcal{M}(\Delta) = \{1, x_4,x_5,x_4^2,x_4x_5,x_4^2x_5\}$. 
\item[Step 1:]
\item For $6 \leq i \leq \eta_B+5$:
\begin{itemize}
\item Introduce vertex $v_i$ to $\mathcal{C}(\Delta)$ through a B-move whose boundary is glued along the cycle $v_1v_2v_3v_4$.
\item Add $\{x_i,x_4x_i,x_i^2,x_4x_i^2\}$ to $\mathcal{M}(\Delta)$.
\end{itemize}

\item For $\eta_B+6 \leq i \leq \eta_A + \eta_B+5$:
\begin{itemize}
\item Introduce vertex $v_i$ to $\mathcal{C}(\Delta)$ through an A-move whose boundary is glued along the cycle $v_1v_2v_3$, 
\item Add $\{x_i,x_i^2,x_i^3\}$to $\mathcal{M}(\Delta)$.
\end{itemize}
\item[Step 2:]
\item For $1 \leq \ell \leq \eta_E$: 
\begin{itemize}
\item Let $\{v_i,v_j\}$ be the revlex smallest missing edge in $\mathcal{C}(\Delta)$.
\item Add the edge $\{v_i,v_j\}$ to $\mathcal{C}(\Delta)$ through an E-move whose boundary is the cycle $v_1v_iv_2v_j$.
\item If $\{v_i,v_j\} = \{v_4,v_5\}$, add $\{x_5^2,x_5^3\}$ to $\mathcal{M}(\Delta)$.
\item Else, add $\{x_ix_j,x_i^2x_j\}$ to $\mathcal{M}(\Delta)$.
\end{itemize}
\end{itemize}

\end{algorithm}

The proof that this algorithm terminates and produces a pure multicomplex is identical to the proof that Steps 0, 1, and 2 in Algorithm \ref{bipyramidbasealg} terminate and produce a pure multicomplex.  Upon the completion of Step 1, the number of missing edges in $\mathcal{C}(\Delta)$ is $\binom{\eta_A+\eta_B}{2}  + 2\eta_A + \eta_B + 1$.  Indeed, none of the newly introduced vertices are adjacent, giving $\binom{\eta_A+\eta_B}{2}$ missing edges; each vertex $v_j$ with $j > 5$ is part of a missing edge with $v_5$, and it is part of a missing edge with $v_4$ if it was introduced through an A-move; finally, the edge $\{v_4,v_5\}$ is missing.  On the other hand, the number of missing degree-$2$ monomials in $\mathcal{M}(\Delta)$ upon the completion of Step 1 is 
\begin{eqnarray*}
\binom{2+\eta_A+\eta_B+1}{2} - \eta_A-2\eta_B-2 &=& \binom{\eta_A+\eta_B}{2} + (3\eta_A+3\eta_B+3) - \eta_A-2\eta_B-2 \\
&=& \binom{\eta_A+\eta_B}{2} + 2\eta_A + \eta_B + 1
\end{eqnarray*}
Once again, this guarantees that Step 2 will terminate.  Moreover, $\{v_4,v_5\}$ is the revlex smallest missing edge in $\mathcal{C}(\Delta)$ after the completion of Step 1.  This means that if $\eta_E > 0$, then the first missing edge that is inserted also adds $x_5^2$ to $\mathcal{M}(\Delta)$, at which point $x_j^2 \in \mathcal{M}(\Delta)$ for all $j$.  This guarantees that the monomials added to $\mathcal{M}(\Delta)$ in subsequent iterations of Step 2 preserve the property that $\mathcal{M}(\Delta)$ is a pure multicomplex.  Thus, we have proved the following theorem. 

\begin{theorem} \label{thm:basespherebipyramid-F=0}
The set of monomials $\mathcal{M}(\Delta)$ output by Algorithm \ref{bipyramidbasealg} is a pure multicomplex.  Moreover $F(\mathcal{M}(\Delta)) = h(\mathcal{C}(\Delta)) = h(\Delta)$.  Consequently, if $\Delta$ is a 2-dimensional PS ear-decomposable simplicial complex with $\Sigma_0 = \partial \sigma^1*\partial\sigma^2$ and $\eta_F = 0$, then $h(\Delta)$ is a pure $\mathcal{O}$-sequence.
\end{theorem}

Theorems \ref{thm:basespherebipyramid-F>0} and \ref{thm:basespherebipyramid-F=0} together handle the case that $\Sigma_0 = \partial\sigma^1 * \partial \sigma^2$. 

\begin{theorem} \label{thm:basespherebipyramid}
Let $\Delta = \Sigma_0 \cup \Sigma_1 \cup \cdots \cup \Sigma_t$ be a $2$-dimensional PS ear-decomposable simplicial complex with $\Sigma_0 = \partial \sigma^1 * \partial \sigma^2$.  Then $h(\Delta)$ is a pure $\mathcal{O}$-sequence.
\end{theorem}

\subsection{The case that $\Sigma_0 = \partial \sigma^1 * \partial \sigma^1 * \partial \sigma^1$}

The case that $\Sigma_0 = \partial \sigma^1 * \partial \sigma^1 * \partial \sigma^1$ is the boundary complex of the octahedron is more complicated than the previous two cases and cannot immediately be reduced to the case that $\Sigma_0 = \partial \sigma^3$.  This case requires special handling of base cases for small values of $\eta_A$, $\eta_B$,  $\eta_E$, and $\eta_F$. We can start by labeling the vertices of the octahedron as in Figure \ref{fig:base-sphere-labels}.

\subsubsection{The case that $\eta_F = 0$}

\begin{algorithm} \label{alg:octahedron-F=0}
Let $\Delta$ be a $2$-dimensional PS ear-decomposable simplicial complex with $\Sigma_0 = \partial \sigma^1 *  \partial \sigma^1 *  \partial \sigma^1$ and $\eta_F = 0$. We define a new PS ear-decomposable complex $\mathcal{C}(\Delta)$ and a set of monomials $\mathcal{M}(\Delta)$ in the variables $\{x_4,x_5, x_6, \ldots,x_{\eta_A+\eta_B+6}\}$ inductively as follows. 

\begin{itemize}
\item[Step 0:]
\item Set $\Sigma_0 = \partial \sigma^1 *  \partial \sigma^1 *  \partial \sigma^1$ on vertex set $\{v_1,v_2,v_3,v_4,v_5, v_6\}$ as shown in Figure \ref{fig:base-sphere-labels}.

\item Set $\mathcal{M}(\Delta) = \{1, x_4,x_5,x_6,x_4x_5,x_4x_6,x_5x_6,x_4x_5x_6\}$.
\item[Step 1:]
\item For $7 \leq i \leq \eta_B+6$: 
\begin{itemize}
\item Introduce vertex $v_i$ to $\mathcal{C}(\Delta)$ through a B-move whose boundary is glued along the cycle $v_1v_2v_4v_3$.
\item Add $\{x_i,x_4x_i,x_i^2,x_4x_i^2\}$ to $\mathcal{M}(\Delta)$.
\end{itemize}

\item For $\eta_B+7 \leq i \leq \eta_A+\eta_B+6$:
\begin{itemize}
\item Introduce vertex $v_i$ to $\mathcal{C}(\Delta)$ through an A-move whose boundary is glued along the cycle $v_1v_2v_3$.
\item Add $\{x_i,x_i^2,x_i^3\}$ to $\mathcal{M}(\Delta)$.
\end{itemize}

\item[Step 2:]
\item For $1 \leq \ell \leq \eta_E$: 
\begin{itemize}
\item Let $\{v_i,v_j\}$ be the revlex smallest missing edge in $\mathcal{C}(\Delta)$.  
\item If $\{v_i,v_j\} = \{v_1,v_4\}$:
\begin{itemize}
\item Add edge $\{v_1,v_4\}$ to $\mathcal{C}(\Delta)$ through an E-move whose boundary is $v_1v_2v_4v_5$.
\item Add $\{x_4^2,x_4^3\}$ to $\mathcal{M}(\Delta)$.
\end{itemize}

\item Else if $\{v_i,v_j\} = \{v_2,v_5\}$:
\begin{itemize}
\item Add edge $\{v_2,v_5\}$ to $\mathcal{C}(\Delta)$ through an E-move whose boundary is $v_1v_2v_4v_5$.
\item Add $\{x_5^2,x_5^3\}$ to $\mathcal{M}(\Delta)$.
\end{itemize}

\item Else if $\{v_i,v_j\} = \{v_3,v_6\}$:
\begin{itemize}
\item Add edge $\{v_3,v_3\}$ to $\mathcal{C}(\Delta)$ through an E-move whose boundary is $v_1v_3v_4v_6$.
\item Add $\{x_6^2,x_6^3\}$ to $\mathcal{M}(\Delta)$.
\end{itemize}

\item Else:
\begin{itemize}
\item Add edge $\{v_i,v_j\}$ to $\mathcal{C}(\Delta)$ through an E-move whose boundary is $v_1v_iv_2v_j$.
\item Add $\{x_ix_j, x_i^2x_j\}$ to $\mathcal{M}(\Delta)$.
\end{itemize}

\end{itemize}
\end{itemize}
\end{algorithm}

Step 0 and Step 1 in Algorithm \ref{alg:octahedron-F=0} proceed as they have in all other cases, and it is clear that $\mathcal{M}(\Delta)$ is a pure multicomplex upon the completion of Step 1. In Step 2, the revlex smallest missing edges $\{v_1,v_4\}$, $\{v_2,v_5\}$, and $\{v_3,v_6\}$ correspond to the values $\ell = 1$, $\ell = 2$, and $\ell = 3$ respectively.  When $\ell \geq 4$, we see that $x_i^2 \in \mathcal{M}(\Delta)$ for all $4 \leq i \leq \eta_A + \eta_B + 6$ and moreover that the revlex smallest missing edge $\{v_i,v_j\}$ satisfies $j \geq 7$ and $i \geq 4$.  Therefore the contribution of the monomials $\{x_ix_j, x_i^2x_j\}$ preserves the property that $\mathcal{M}(\Delta)$ is a pure multicomplex.  This proves the following theorem. 

\begin{theorem} \label{thm:octahedron-F=0}
The set of monomials $\mathcal{M}(\Delta)$ output by Algorithm \ref{alg:octahedron-F=0} is a pure multicomplex.  Moreover $F(\mathcal{M}(\Delta)) = h(\mathcal{C}(\Delta)) = h(\Delta)$.  Consequently, if $\Delta$ is a 2-dimensional PS ear-decomposable simplicial complex with $\Sigma_0 = \partial \sigma^1*\partial\sigma^1*\partial\sigma^1$ and $\eta_F = 0$, then $h(\Delta)$ is a pure $\mathcal{O}$-sequence.
\end{theorem}

\subsubsection{The case that $\eta_E = 0$}

\begin{lemma} \label{lemma:octahedron-E=0}
Let $\Delta$ be a $2$-dimensional PS ear-decomposable simplicial complex with  $\Sigma_0 = \partial \sigma^1 * \partial \sigma^1 * \partial \sigma^1$ and $\eta_E = 0$.  If $\eta_B = 0$, then $\eta_F = 0$.  Otherwise, if $\eta_B > 0$, then 
\begin{displaymath}
\eta_F \leq 
\begin{cases}
2\eta_B-1 & \text{if } \eta_A = 0, \\
2\eta_B & \text{if } \eta_A > 0.
\end{cases}
\end{displaymath}
\end{lemma}

\begin{proof}
Attaching a PS ball by an A-move does not change the number of missing triangles in a PS ear-decomposable simplicial complex.  Therefore, since $\Sigma_0$ does not contain any missing triangles, neither will $\Delta$ if $\eta_B = 0$.  

Now suppose $\eta_B > 0$.  When a new vertex is introduced through a type-B move, there is a potential for two missing triangles to be introduced to $\Delta$.  Specifically, if vertex $v$ is introduced through a B-move whose boundary vertices are labeled as shown below, then the new missing triangles that are created are $\{v,x,z\}$ (as long as $\{x,z\}$ was already an edge) and $\{v,w,y\}$ (as long as $\{w,y\}$ was already an edge).

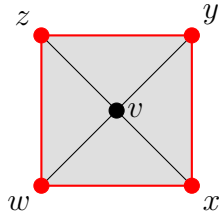
\begin{center}
\begin{tikzpicture}
\draw[red, thick , fill=gray!25]  (0,0) -- (2,0) -- (2,2) -- (0,2) -- (0,0);
\foreach \p in {(0,0), (2,0), (2,2), (0,2)}{
	\draw (1,1) -- \p;
	\draw[red, fill=red] \p circle (.1); 
}
\draw (1,1) node[anchor=west] {$v$};
\draw (0,0) node[anchor=north east] {$w$};
\draw (2,0) node[anchor=north west] {$x$};
\draw (2,2) node[anchor=south west] {$y$};
\draw (0,2) node[anchor=south east] {$z$};
\draw[fill=black] (1,1) circle (.1);

\end{tikzpicture}
\end{center}

Therefore, $\eta_F \leq 2\eta_B$ whenever $\eta_B > 0$.  However, if $\eta_A = 0$, then the first ear attached through a B-move cannot contribute two missing triangles because the graph of the octahedron does not contain an induced $K_4$ subgraph (meaning either $\{x,z\}$ or $\{w,y\}$ will be missing).  This means that when $\eta_A = 0$ and $\eta_B>0$, it must be the case that $\eta_F \leq 2\eta_B-1$. 
\end{proof}

\begin{algorithm} \label{alg:octahedron-E=0}
Let $\Delta$ be a $2$-dimensional PS ear-decomposable simplicial complex with $\Sigma_0 = \partial \sigma^1 *  \partial \sigma^1 *  \partial \sigma^1$ and $\eta_E = 0$. We define a new PS ear-decomposable complex $\mathcal{C}(\Delta)$ and a set of monomials $\mathcal{M}(\Delta)$ in the variables $\{x_4,x_5, x_6, \ldots,x_{\eta_A+\eta_B+6}\}$ inductively as follows. 

\begin{itemize}
\item[Step 0:]
\item Set $\Sigma_0 = \partial \sigma^1 *  \partial \sigma^1 *  \partial \sigma^1$ on vertex set $\{v_1,v_2,v_3,v_4,v_5, v_6\}$ as shown above.

\item Set $\mathcal{M}(\Delta) = \{1, x_4,x_5,x_6,x_4x_5,x_4x_6,x_5x_6,x_4x_5x_6\}$.
\item[Step 1:]
\item If $\eta_A = 0$ and $\eta_B > 0$:
\begin{itemize}
\item Introduce vertex $v_7$ to $\mathcal{C}(\Delta)$ through a $B$-move whose boundary is glued along the cycle $v_1v_2v_4v_3$.
\item Add $\{x_7,x_4x_7,x_7^2,x_4x_7^2\}$ to $\mathcal{M}(\Delta)$
\item For $8 \leq i \leq 7+(\eta_B-1)$: 
\begin{itemize}
\item Introduce vertex $v_i$ to $\mathcal{C}(\Delta)$ through a B-move whose boundary is glued along the cycle $v_1v_2v_3v_7$.
\item Add $\{x_i,x_i^2,x_7x_i,x_7x_i^2\}$ to $\mathcal{M}(\Delta)$.
\end{itemize}
\end{itemize}

\item If $\eta_A > 0$:
\begin{itemize}
\item For $7 \leq i \leq \eta_A+6$:
\begin{itemize}
\item Introduce vertex $v_i$ to $\mathcal{C}(\Delta)$ through an A-move whose boundary is glued along the cycle $v_1v_2v_3$.
\item Add $\{x_i,x_i,x_i^3\}$ to $\mathcal{M}(\Delta)$.
\end{itemize}
\item For $\eta_A + 7 \leq i \leq \eta_A+\eta_B+6$:
\begin{itemize}
\item Introduce vertex $v_i$ to $\mathcal{C}(\Delta)$ through a B-move whose boundary is glued along the cycle $v_1v_2v_3v_7$.
\item Add $\{x_i,x_i^2,x_7x_i,x_7x_i^2\}$ to $\mathcal{M}(\Delta)$.
\end{itemize}
\end{itemize}

\item[Step 2:]
\item If $\eta_A = 0$: 
\begin{itemize}
\item Let $\mathcal{S} = \{x_7^3\} \cup \{x_7^2x_i, x_i^3 \ : \ 8 \leq i \leq \eta_B+6\}$.
\item Let $\mathcal{F} = \{\{v_1,v_4,v_7\}\} \cup \{\{v_1,v_3,v_i\},\{v_2,v_7,v_i\} \ : \ 8 \leq i \leq \eta_B+6\}$. 
\item Add the first $\eta_F$ monomials in $\mathcal{S}$ (under revlex order) to $\mathcal{M}(\Delta)$.
\item Add the first $\eta_F$ faces in $\mathcal{F}$ (under revlex order) to $\mathcal{C}(\Delta)$.
\end{itemize}

\item If $\eta_A>0$:
\begin{itemize}
\item Let $\mathcal{S} = \{x_7^2x_i, x_i^3 \ : \ \eta_A+7 \leq i \leq \eta_A+\eta_B+6\}$.
\item Let $\mathcal{F} =  \{\{v_1,v_3,v_i\},\{v_2,v_7,v_i\} \ : \ \eta_A+7 \leq i \leq \eta_A+\eta_B+6\}$. 
\item Add the first $\eta_F$ monomials in $\mathcal{S}$ (under revlex order) to $\mathcal{M}(\Delta)$.
\item Add the first $\eta_F$ faces in $\mathcal{F}$ (under revlex order) to $\mathcal{C}(\Delta)$.
\end{itemize}

\end{itemize}
\end{algorithm}

This algorithm requires more case analysis because of the bound on $\eta_F$ when $\eta_A = 0$ and $\eta_B > 0$ in Lemma \ref{lemma:octahedron-E=0}. In Step 1, we introduce all vertices through A- and B-moves.  In the case that $\eta_A > 0$, we break with our previous strategy and introduce all type-A vertices first out of convenience.  As noted in the proof of Lemma \ref{lemma:octahedron-E=0}, we do this so that there will be a $K_4$ subgraph on the vertices $\{v_1,v_2,v_3,v_7\}$ where the boundaries of the type-B balls can be glued.  

In Step 2, when $\eta_A = 0$, there are $2\eta_B-1$ faces in set $\mathcal{F}$, which correspond to the maximum number of missing triangles in a PS ear-decomposable complex with $\eta_E = 0$ and $\eta_A = 0$ by Lemma \ref{lemma:octahedron-E=0}.  There are $2\eta_B-1$ corresponding monomials in the set $\mathcal{S}$ that can be added to $\mathcal{M}(\Delta)$.  Similarly, when $\eta_A > 0$, there are $2\eta_B$ faces in set $\mathcal{F}$ and $2\eta_B$ monomials in set $\mathcal{S}$.  Therefore, Algorithm \ref{alg:octahedron-E=0} terminates and produces a pure $\mathcal{O}$-sequence.  This proves the following theorem. 

\begin{theorem} \label{thm:octahedron-E=0}
The set of monomials $\mathcal{M}(\Delta)$ output by Algorithm \ref{alg:octahedron-E=0} is a pure multicomplex.  Moreover $F(\mathcal{M}(\Delta)) = h(\mathcal{C}(\Delta)) = h(\Delta)$.  Consequently, if $\Delta$ is a 2-dimensional PS ear-decomposable simplicial complex with $\Sigma_0 = \partial \sigma^1*\partial\sigma^1*\partial\sigma^1$ and $\eta_E = 0$, then $h(\Delta)$ is a pure $\mathcal{O}$-sequence.
\end{theorem}

\subsubsection{The case that $\eta_E > 0$ and $\eta_F > 0$}

Our goal in the case that $\eta_E > 0$ and $\eta_F>0$ is to reduce to the cases already established in which $\Sigma_0$ is the boundary of a tetrahedron or a bipyramid. 

\begin{proposition} \label{prop:base-oct-withv1v4}
Let $\Delta = \Sigma_0 \cup \Sigma_1 \cup \cdots \cup \Sigma_t$ be a $2$-dimensional PS ear-decomposable simplicial complex with $\Sigma_0 = \partial \sigma^1 *  \partial \sigma^1 *  \partial \sigma^1$.  Assume the vertices of $\Sigma_0$ are labeled as in Figure \ref{fig:base-sphere-labels}.  If $\{v_1,v_4\} \in \Delta$, then there exists a PS ear-decomposable simplicial complex $\Delta' = \Sigma_0' \cup \Sigma_1' \cup \cdots \cup \Sigma_t'$ such that $\Sigma_0' = \partial\sigma^2 * \partial \sigma^1$ and $h(\Delta) = h(\Delta')$.  Consequently, $h(\Delta)$ is a pure $\mathcal{O}$-sequence.
\end{proposition}

\begin{proof}
Let $\Sigma_i$ be the ear in which edge $\{v_1,v_4\}$ is introduced.  The ear $\Sigma_i$ must be of type-E, so there are vertices $u$ and $w$ such that $\Sigma_i$ has triangular faces $\{v_1,v_4,u\}$ and $\{v_1,v_4,w\}$. 

The proof proceeds in two steps. First we handle the case that $\{u,w\} = \{v_2,v_5\}$.  Next we reduce to that case. 

If $\{u,w\} = \{v_2,v_5\}$, let $\Sigma_0'$ be the bipyramid and let $\Sigma_1'$ be the PS ear shown in Figure \ref{new-ears}. 

\begin{center}
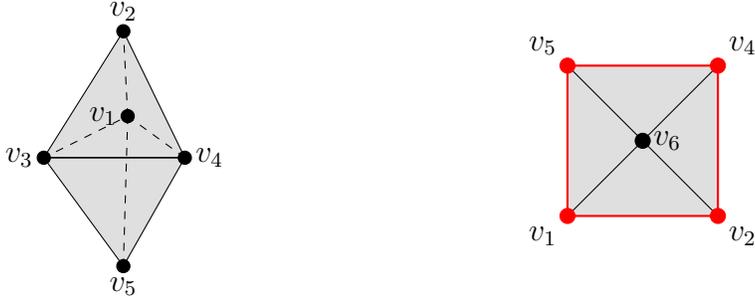
\begin{figure}[h]
\begin{tabular}{p{.4\textwidth}p{.4\textwidth}}
\begin{center}
\begin{tikzpicture}[scale=1.25]
\def \ts {(.25,1.25,-.25)};
\def \bs{(.25,-1.25,-.25)};
\def \a{(1,0,0)};
\def \b{(-.5,0,0)};
\def \c{(-.05,0,-1.15)};
\draw[fill=gray!25] \ts node[anchor=south] {$v_2$}-- \a node[anchor=west] {$v_4$} -- \b node[anchor=east] {$v_3$} -- \ts;
\draw[fill=gray!25] \bs node[anchor=north]{$v_5$} -- \a -- \b -- \bs;
\draw[fill=black] \a circle (.07);
\draw[fill=black] \b circle (.07);
\draw[fill=black] \c circle (.07) node[anchor=east] {$v_1$};
\draw[fill=black]  \ts circle (.07);
\draw[fill=black] \bs circle (.07);
\draw[dashed] \c -- \ts;
\draw[dashed] \c -- \bs;
\draw \a -- \b;
\draw[dashed] \a -- \c  -- \b;
\end{tikzpicture}
\end{center}
&
\begin{center}
\begin{tikzpicture}
\draw[red, thick , fill=gray!25]  (0,0) -- (2,0) -- (2,2) -- (0,2) -- (0,0);
\foreach \p in {(0,0), (2,0), (2,2), (0,2)}{
	\draw (1,1) -- \p;
	\draw[red, fill=red] \p circle (.1); 
}
\draw (1,1) node[anchor=west] {$v_6$};
\draw (0,0) node[anchor=north east] {$v_1$};
\draw (2,0) node[anchor=north west] {$v_2$};
\draw (2,2) node[anchor=south west] {$v_4$};
\draw (0,2) node[anchor=south east] {$v_5$};
\draw[fill=black] (1,1) circle (.1);

\end{tikzpicture}
\end{center}
\end{tabular}
\caption{The new PS sphere $\Sigma_0'$ (left) and PS ball $\Sigma_1'$ (right).}
 \label{new-ears}
\end{figure}
\end{center}

Note that $\Sigma_0' \cup \Sigma_1' = \Sigma_0 \cup \Sigma_i$.  Therefore, we can decompose $\Delta$ as $\Delta = \Sigma_0' \cup \Sigma_1' \cup \Sigma_2 \cup \cdots \cup \Sigma_{i-1} \cup \Sigma_{i+1} \cup \cdots \cup \Sigma_t$, which satisfies the claim.  

Now suppose instead that $\{u,w\} \neq \{v_2,v_5\}$. First, assume $\{u,w\}$ and $\{v_2,v_5\}$ share one element.  Without loss of generality, $u = v_2$ but $w \neq v_5$.  Consider the set $\tau = \{v_1,v_4,v_5\}$.  The three edges $\{v_1,v_4\}$, $\{v_1,v_5\}$, and $\{v_4,v_5\}$ all belong to $\Delta$.  

If $\tau \in \Delta$, then $\tau$ is added to $\Delta$ through an F-move because $\{v_1,v_5\}$ and $\{v_4,v_5\}$ are edges in $\Sigma_0$ and $\{v_1,v_4\}$ is introduced through $\Sigma_i$, and neither $\Sigma_0$ nor $\Sigma_i$ contains $\tau$. Therefore, there is some $j > i$ such that $\Sigma_j$ is an F-move introducing face $\tau$.  We can replace $\Sigma_i$ with the PS ear $\Sigma_i'$, which is an E-move introducing edge $\{v_1,v_4\}$ using the triangles $\{v_1,v_2,v_4\}$ and $\{v_1,v_4,v_5\}$, and replace $\Sigma_j$ with the PS ear $\Sigma_j'$, which is an F-move introducing the face $\{v_1,v_4,w\}$. This is an alternate PS ear decomposition of $\Delta$ that satisfies the claim.

On the other hand, if $\tau \notin \Delta$,  let $\Delta'$ be the complex obtained from $\Delta$ by removing the face $\{v_1,v_4,w\}$ and adding the face $\tau$.  This complex is PS ear-decomposable because we can replace $\Sigma_i$ with the ear $\Sigma_i'$, which is an E-move with triangles $\{v_1,v_2,v_4\}$ and $\{v_1,v_4,v_5\}$.  Clearly $h(\Delta') = h(\Delta)$, so $\Delta'$ satisfies the claim. 

Finally, if $\{u,w\} \cap \{v_2,v_5\} = \emptyset$, we apply the previous argument twice. 
\end{proof}

Now we turn our attention to the case that $\{v_1,v_4\} \notin \Delta$.  We begin with a lemma whose proof is essentially Algorithm \ref{tetrabasealg}.

\begin{lemma} \label{compress-constructible}
Let $G$ be a constructible graph and $\mathcal{C}(G)$ its compression. There exists a $2$-dimensional PS ear-decomposable simplicial complex $\Delta$ whose underlying graph is $\mathcal{C}(G)$.
\end{lemma}

\begin{proof}
We begin with the boundary of a tetrahedron on vertices $\{v_1,v_2,v_3,v_4\}$.  For each vertex $v_i$ of type-B in $\mathcal{C}(G)$, we attach an ear of type-B whose boundary is the cycle $v_1v_2v_3v_4$, and for each vertex of type-A, we attach an ear of type-A whose boundary is the cycle $v_1v_2v_3$.  If $\{v_i,v_j\}$ is an edge added to $\mathcal{C}(G)$, then $v_1$ and $v_2$ are common neighbors of $v_i$ and $v_j$ by construction, so we attach an ear of type-E whose boundary is the cycle $v_1v_iv_2v_j$.  The underlying graph of the resulting complex is $\mathcal{C}(G)$.
\end{proof}

For the remainder of the paper, we assume $\Delta$ is a $2$-dimensional PS ear-decomposable simplicial complex with $\Sigma_0 = \partial \sigma^1 *  \partial \sigma^1 *  \partial \sigma^1$ such that $\{v_1,v_4\} \notin \Delta$.  Let $G$ be the underlying graph of $\Delta$. Let $G'$ be the graph obtained from $G$ by applying the shifting operator $\mathcal{S}_{4,5}$. We know $G$ is constructible and claim that $G'$ is also constructible. First,  $\mathcal{S}_{4,5}$ acts on the graph of $\Sigma_0$ by removing edge $\{v_1,v_5\}$ and adding edge $\{v_1,v_4\}$. 

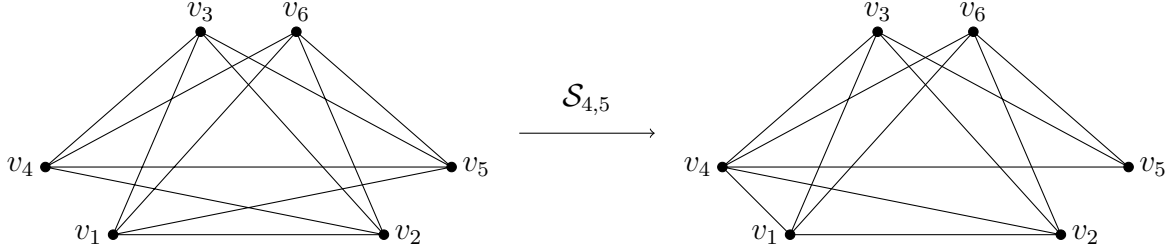
\begin{center}
\begin{tikzpicture}[scale=.9]
\def \a{(-2,0)};
\def \b{(-3,1)};
\def \c{(2,0)};
\def \d{(3,1)};
\def \e{(-.707,3)};
\def \f{(.707,3)};

\filldraw[black] \a circle (.07) node[anchor=east] {$v_1$};
\filldraw[black] \c circle (.07) node[anchor=west] {$v_2$};
\filldraw[black] \e circle (.07) node[anchor=south] {$v_3$};
\filldraw[black] \b circle (.07) node[anchor=east] {$v_4$};
\filldraw[black] \d circle (.07) node[anchor=west] {$v_5$};
\filldraw[black] \f circle (.07) node[anchor=south] {$v_6$};

\draw \a -- \c -- \b -- \d -- \a -- \e -- \b -- \f -- \a;
\draw \c -- \e -- \d -- \f -- \c;

\draw[->] (4,1.5) -- (6,1.5);
\draw (5,2) node {$\mathcal{S}_{4,5}$};

\def \a{(8,0)};
\def \b{(7,1)};
\def \c{(12,0)};
\def \d{(13,1)};
\def \e{(10-.707,3)};
\def \f{(10.707,3)};

\filldraw[black] \a circle (.07) node[anchor=east] {$v_1$};
\filldraw[black] \c circle (.07) node[anchor=west] {$v_2$};
\filldraw[black] \e circle (.07) node[anchor=south] {$v_3$};
\filldraw[black] \b circle (.07) node[anchor=east] {$v_4$};
\filldraw[black] \d circle (.07) node[anchor=west] {$v_5$};
\filldraw[black] \f circle (.07) node[anchor=south] {$v_6$};

\draw  \a -- \e -- \b -- \f -- \a;
\draw \c -- \e -- \d -- \f -- \c;
\draw \b -- \c -- \a -- \b -- \d;
\end{tikzpicture}
\end{center}

The resulting graph is constructible, starting with the complete graph on $\{v_1,v_2,v_3,v_4\}$, then adding edges $\{v_1,v_6\}$, $\{v_3,v_6\}$, and $\{v_4,v_6\}$ through an A-move that introduces vertex $v_6$, and adding edges $\{v_2,v_5\}$, $\{v_4,v_5\}$, and $\{v_5,v_6\}$ through an A-move that introduces vertex $v_5$. We now proceed inductively.  For each vertex $v_{\ell}$ with $\ell > 6$, consider the $\ell$-labeled edges incident to $v_{\ell}$ in $G$.  Let $N$ denote the set of vertices incident to $v_\ell$ by such edges.  If $v_5 \notin N$ or if $\{v_4,v_5\} \subset N$, add $\ell$-labeled edges $\{v_i,v_\ell\}$ to $G'$ for all $i \in G$. Otherwise, if $v_5 \in N$ but $v_4 \notin N$ add $\ell$-labeled edges $\{v_4,v_\ell\}$, and $\{v_i,v_n\}$ for $i \in N \setminus \{v_5\}$ to $G'$.

Having done this, we move on to consider the unlabeled edges in $G$.  Suppose $\{v_i,v_j\}$ is such an edge and $i<j$. If $i>5$, add the unlabeled edge $\{v_i,v_j\}$ to $G'$ as well.  If $i=4$, it is possible that $\{v_4,v_j\}$ already belongs to $G'$ as a $j$-labeled edge.  If this happens, then $\{v_4,v_j\}$ is an unlabeled edge in $G$ and $\{v_5,v_j\}$ is a $j$-labeled edge in $G$.  This means $\{v_5,v_j\} \notin G'$, so we can add the edge $\{v_5,v_j\}$ to $G'$. Finally, if $i=5$ and $\{v_4,v_j\}$ is not an edge in $G'$, add the unlabeled edge $\{v_4,v_j\}$ to $G'$. Otherwise, add the unlabeled edge $\{v_5,v_j\}$ to $G'$. This proves that $G'$ is also constructible. 

Next, by Lemma \ref{lemma:shifttriangles} we know $\#T(G) \leq \#T(G')$, and because $G'$ is constructible, Theorem \ref{thm:compressiontriangles} tells us $\#T(G') \leq \#T(\mathcal{C}(G'))$.  Finally, by Lemma \ref{compress-constructible}, there is a PS ear-decomposable simplicial complex $\Delta'$ whose underlying graph is $\mathcal{C}(G')$.  Use $\eta_A'$, $\eta_B'$, $\eta_E'$, and $\eta_F'$ to denote the number of ears of type A, B, E, and F respectively in the construction of $\Delta'$. It follows from the proof of Lemma \ref{compress-constructible} that the PS sphere used in the construction of $\Delta'$ is $\partial \sigma^3$ and that $\eta_A' = \eta_A+2$, $\eta_B' = \eta_B$,  $\eta_E' = \eta_E$, and $\eta_F' = 0$.  Finally, since $\#T(G) \leq \#T(\mathcal{C}(G'))$, there are at least $\eta_F$ missing triangles in $\Delta'$ that can be filled.  This proves the following theorem. 

\begin{theorem} \label{octahedron-to-tetrahedron}
Let $\Delta = \Sigma_0 \cup \Sigma_1 \cup \cdots \cup \Sigma_t$ be a $2$-dimensional PS ear-decomposable simplicial complex with $\Sigma_0 = \partial \sigma^1 *  \partial \sigma^1 *  \partial \sigma^1$ such that $\{v_1,v_4\} \notin \Delta$. 

There exists a PS ear-decomposable simplicial complex $\Delta' = \Sigma_0' \cup \Sigma_1' \cup \cdots \cup \Sigma_{t+2}'$ such that 
\begin{itemize}
\item $\Sigma_0' = \partial \sigma^3$,
\item $h(\Delta') = h(\Delta) + (0,0,0,2)$,
\item $\eta_A' = \eta_A+2$,
\item $\eta_B' = \eta_B$,
\item $\eta_E' = \eta_E$, and
\item $\eta_F' = \eta_F$.
\end{itemize}
\end{theorem}

This brings us to the final subcase. 

\begin{theorem} \label{octahedron-remove-extra-monomials}
Let $\Delta = \Sigma_0 \cup \Sigma_1 \cup \cdots \cup \Sigma_t$ be a $2$-dimensional PS ear-decomposable simplicial complex with $\Sigma_0 = \partial \sigma^1 *  \partial \sigma^1 *  \partial \sigma^1$ such that $\{v_1,v_4\} \notin \Delta$. Assume further that $\eta_E > 0$ and $\eta_F > 0$. Then $h(\Delta)$ is a pure $\mathcal{O}$-sequence.
\end{theorem}

\begin{proof}
Let $\Delta'$ be the complex whose existence is guaranteed by Theorem \ref{octahedron-to-tetrahedron}, and let $\mathcal{C}(\Delta')$ and $\mathcal{M} = \mathcal{M}(\Delta')$ be the PS ear-decomposable simplicial complex and corresponding pure multicomplex output by Algorithm \ref{tetrabasealg}. It is important to recall that the vertex labels in $\mathcal{C}(\Delta')$ may be permuted from their initial labels in $\Delta$ and $\Delta'$.  Our goal is to show that there are two degree-$3$ monomials in $\mathcal{M}$ that can be removed without destroying the purity of the multicomplex.  We examine three cases. 

\textit{Case 1:}  Vertices $v_5$ and $v_6$ have type A in $\mathcal{C}(\Delta')$. 

In Step 1 of Algorithm \ref{tetrabasealg}, the monomials $\{x_5,x_5^2,x_5^3\}$ and $\{x_6,x_6^2,x_6^3\}$ are added to the initial set of monomials $\{1,x_4,x_4^2,x_4^3\}$ in $\mathcal{M}$.  Because $\eta_E > 0$, and $x_4x_5$ is the revlex smallest monomial of degree-2 that does not belong to $\mathcal{M}$ upon the completion of Step 1, the monomials $\{x_4x_5,x_4^2x_5\}$ will be added to $\mathcal{M}$ in Step 2 of Algorithm \ref{tetrabasealg}. Finally, because $\eta_F>0$, the monomial $x_4x_5^2$ will be added in Step 3 as it is the revlex smallest monomial of degree $3$ that does not belong to $\mathcal{M}$ upon the completion of Step 2.  This means that the following monomials will belong to $\mathcal{M}$ upon the completion of Algorithm \ref{tetrabasealg}

\begin{center}
\begin{tabular}{ccccc}
$x_4^3$ & $x_5^3$  & $x_6^3$  & $x_4^2x_5$  & $x_4x_5^2$ \\
$x_4^2$ & $x_5^2$  & $x_6^2$  & $x_4x_5$  & \\
$x_4$ & $x_5$  & $x_6$  &  & \\
$1$ &  &  &  & \\
\end{tabular}
\end{center}

Now we can observe that $\mathcal{M} \setminus \{x_4^3,x_5^3\}$ is a pure multicomplex whose $F$-vector is $h(\Delta)$.

\textit{Case 2:} Vertex $v_5$ has type B and  vertex $v_6$ has type A in $\mathcal{C}(\Delta')$. 

In Step 1 of Algorithm \ref{tetrabasealg}, the monomials $\{x_5,x_4x_5x_5^2,x_4x_5^2\}$ and $\{x_6,x_6^2,x_6^3\}$ are added to the initial set of monomials $\{1,x_4,x_4^2,x_4^3\}$ in $\mathcal{M}$. Because $\eta_E > 0$, the monomials $\{x_4x_6,x_4^2x_6\}$ will be added to $\mathcal{M}$ in Step 2 of Algorithm \ref{tetrabasealg}. Finally, because $\eta_F>0$, the monomial $x_4^2x_5$ will be added in Step 3.  This means that the following monomials will belong to $\mathcal{M}$ upon the completion of Algorithm \ref{tetrabasealg}

\begin{center}
\begin{tabular}{cccccc}
$x_4^3$ & \multicolumn{2}{c}{$x_4x_5^2$}  & $x_6^3$  & $x_4^2x_6$  & $x_4^2x_5$ \\
$x_4^2$ & $x_4x_5$ & $x_5^2$  & $x_6^2$  & $x_4x_6$  & \\
$x_4$ & \multicolumn{2}{c}{$x_5$} & $x_6$  &  & \\
$1$ &  &  &  & \\
\end{tabular}
\end{center}

Now we can observe that $\mathcal{M} \setminus \{x_4^3,x_4^2x_5\}$ is a pure multicomplex whose $F$-vector is $h(\Delta)$.

\textit{Case 3:}  Vertices $v_5$ and $v_6$ have type B in $\mathcal{C}(\Delta')$. 

In Step 1 of Algorithm \ref{tetrabasealg}, the monomials $\{x_5,x_4x_5x_5^2,x_4x_5^2\}$ and $\{x_6,x_4x_6x_6^2,x_4x_6^2\}$ are added to the initial set of monomials $\{1,x_4,x_4^2,x_4^3\}$ in $\mathcal{M}$.  Because $\eta_E > 0$, the monomials $\{x_5x_6,x_5^2x_6\}$ will be added to $\mathcal{M}$ in Step 2 of Algorithm \ref{tetrabasealg}. Finally, because $\eta_F>0$, the monomial $x_4^2x_5$ will be added in Step 3.  This means that the following monomials will belong to $\mathcal{M}$ upon the completion of Algorithm \ref{tetrabasealg}

\begin{center}
\begin{tabular}{ccccccc}
$x_4^3$ & \multicolumn{2}{c}{$x_4x_5^2$}   &  \multicolumn{2}{c}{$x_4x_6^2$}   & $x_5^2x_6$  & $x_4^2x_5$ \\
$x_4^2$  & $x_4x_5$ & $x_5^2$  & $x_4x_6$ & $x_6^2$  & $x_5x_6$ &   \\
$x_4$ & \multicolumn{2}{c}{$x_5$}  & \multicolumn{2}{c}{$x_6$}  &  & \\
$1$ &  &  &  & & & \\
\end{tabular}
\end{center}

Now we can observe that $\mathcal{M} \setminus \{x_4^3,x_4x_5^2\}$ is a pure multicomplex whose $F$-vector is $h(\Delta)$.
\end{proof}

Together, 
 Theorem \ref{thm:octahedron-F=0},  Theorem \ref{thm:octahedron-E=0}, Proposition \ref{prop:base-oct-withv1v4}, and Theorems \ref{octahedron-to-tetrahedron} and  \ref{octahedron-remove-extra-monomials} exhaust all possibilities when $\Sigma_0 = \partial \sigma^1 *\partial \sigma^1 *\partial \sigma^1.$  This completes the proof of the main theorem. 
 
\begin{theorem} \label{thm:basesphereoctahedron}
Let $\Delta = \Sigma_0 \cup \Sigma_1 \cup \cdots \cup \Sigma_t$ be a $2$-dimensional PS ear-decomposable simplicial complex with $\Sigma_0 = \partial \sigma^1 * \partial \sigma^1 * \partial\sigma^1$.  Then $h(\Delta)$ is a pure $\mathcal{O}$-sequence.
\end{theorem}

\bibliography{biblio}
\bibliographystyle{plain}

\end{document}